\documentclass{amsart}
\usepackage{amsmath,amsthm}
\usepackage{amsfonts,amssymb}

\usepackage{enumerate}
\usepackage{multirow}

\newtheorem{thm}{Theorem}[section]
\newtheorem{cor}[thm]{Corollary}
\newtheorem{lem}[thm]{Lemma}
\newtheorem{prop}[thm]{Proposition}

\theoremstyle{remark}
\newtheorem{rem}{Remark}[section]

 \def\a{{\alpha}} 
 \def\b{{\beta}}
 \def\g{{\gamma}}
 
 \def\t{{\theta}}
 
 \def\d{{\delta}}

 \def\la{{\langle}}
 \def\ra{{\rangle}}

 \def\CE{{\mathcal E}}

 \def\CK{{\mathcal K}}
 
 \def\CO{{\mathcal O}}
      
 \def\CU{{\mathcal U}}     
 \def\CS{{\mathcal S}}     
 \def\CV{{\mathcal V}}
 \def\CW{{\mathcal W}}

 \def\RR{{\mathbb R}}

\newcommand{\wh}{\widehat}

\newif\ifpdf
\ifx\pdfoutput\undefined
  \pdffalse
\else
  \pdftrue
\fi

\ifpdf
  \usepackage[pdftex]{graphicx}
  \DeclareGraphicsExtensions{.pdf,.jpg,.png}
\else
  \usepackage{graphicx}
\fi

\begin{document}
 
\title[orthogonal polynomials of two variables]
{Orthogonal polynomials and expansions  for a family of
weight functions in two variables}
\author{Yuan Xu}
\address{Department of Mathematics\\ University of Oregon\\
    Eugene, Oregon 97403-1222.}\email{yuan@math.uoregon.edu}

\date{\today}

\begin{abstract}
Orthogonal polynomials for a family of weight functions on $[-1,1]^2$, 
$$
\CW_{\a,\b,\g}(x,y) = |x+y|^{2\a+1} |x-y|^{2\b+1} (1-x^2)^\g(1-y^2)^{\g},
$$
are studied and shown to be related to the Koornwinder polynomials defined on
the region bounded by two lines and a parabola. In the case of $\g = \pm 1/2$, 
an explicit basis of orthogonal polynomials is given in terms of Jacobi polynomials 
and a closed formula for the reproducing kernel is obtained. The latter is used to 
study the convergence of orthogonal expansions for these weight functions. 
\end{abstract}

\maketitle

\section{Introduction}
\setcounter{equation}{0}
 
Orthogonal polynomials of two variables with respect to a nonnegative weight function
that has all moments finite are known to exist (\cite{DX}).  A basis of orthogonal polynomials   
can be written down, say, in terms of moments, but such a basis is often hard to work with. 
For studying orthogonal polynomials and orthogonal expansions, additional structures are
often called for. In the case of classical weight functions in two variables, for example, 
an orthogonal basis can be expressed in terms of classical orthogonal polynomials of one 
variable. There are, however, not many such examples; each additional one is valuable 
in its own right.

The purpose of the present paper is to study orthogonal polynomials and orthogonal expansions 
with respect to a family of weight functions defined on $[-1,1]^2$, which includes as a 
special case 
\begin{equation} \label{CWabg}
           \CW_{\a,\b,\g}(x,y) := |x-y|^{2\a+1} |x+y|^{2\b+1} (1-x^2)^\g (1-y^2)^{\g}, 
\end{equation}
where $\a,\b,\g > -1$ and $\a + \g +3/2 >0$ and $\b + \g + 3/2 > 0$. In the case 
$\a = \b = -1/2$, $\CW_{\a,\b, \g}$ is the product Gegenbauer weight functions, for which 
an orthogonal basis is given by product Gagenbauer polynomials.  We shall show that 
orthogonal polynomials for this family of weight functions can be expressed in terms of 
orthogonal polynomials in one variable when $\g = \pm \frac12$. Our study starts from a 
realization that it is possible to express the orthogonal polynomials with respect to 
$\CW_{\a,\b,\g}$ in terms of the Koornwinder polynomials that are orthogonal
with respect to the weight function 
\begin{equation} \label{Wabg}
    W_{\a,\b,\g}(u,v) := (1-u+v)^\a (1+u+v)^\b (u^2 - 4 v)^{\g}
\end{equation}
defined on the domain $\Omega$ bounded by two lines and a parabola,
\begin{equation} \label{Omega}
       \Omega: = \{(u,v): 1+u+v > 0, 1-u+v > 0, u^2 > 4v\}. 
\end{equation} 
Orthogonal polynomials with respect to $W_{\a,\b,\g}$ were first studied by 
Koornwinder in \cite{K74}, where an orthogonal basis is uniquely defined and  
shown to consist of eigenfunctions of two differential operators of order 2 and order 4,
respectively. In the case of $\g = \pm \frac12$, the orthogonal polynomials can 
be given in terms of the Jacobi polynomials of one variable. Further studies were 
carried out in \cite{KS, S}; in particular, explicit formula for the orthogonal polynomials 
were derived and various recursive relations were established. The connection to 
orthogonal polynomials with respect to $\CW_{\a,\b,\g}$ is somewhat surprising, 
but simple in retrospect as can be seen by the relation
$$
   \CW_{\a,\b,\g}(x,y) = 4^{-\g} W_{\a,\b,\g}(2xy, x^2+y^2 -1) |x^2 - y^2|. 
$$
As a result of the connection, an orthogonal basis for $\CW_{\a,\b, \pm \frac12}$
can be given explicitly in terms of the Jacobi polynomials. 
 
An explicit orthogonal basis makes it possible to study orthogonal expansions,
for which however it is essential to have access to the reproducing kernel of
the space of polynomials of degree at most $n$ in $L^2(W)$. It turns out that 
closed forms of the reproducing kernels for $W_{\a,\b,\pm \frac12}$ and for 
$\CW_{\a,\b,\pm \frac12}$, respectively, can be given in terms of the reproducing 
kernels of the Jacobi polynomials. This allows us to prove several results on the 
convergence of the orthogonal expansions for these weight functions. The results 
include $L^p$ convergence of the partial sum operators and sharp estimate of 
the Lebesgue constants. It is interesting to note that analogous result on the $L^p$ 
convergence has not been proven for the product Jacobi series on the square (the 
partial sum is defined in terms of polynomial subspace of total order, see Remark 2.1). 
In fact, as far as we know, $W_{\a,\b, -\frac12}$ appears to be the first family of weight functions 
on the unit square for which a comprehensive study of orthogonal expansions is possible. 

The Koornwinder polynomials are derived from the symmetric orthogonal polynomials
with respect to the weight function 
$$
   W_{\a,\b,\g}(x+y,x y)|x-y| = (1-x)^\a (1+x)^\b (1-y)^\a (1+y)^\b | x - y |^{2 \gamma+1},    
$$
which are the generalized Jacboi polynomials of $BC_2$ type. The latter are the 
first case of the generalized Jacobi polynomials of $BC_n$ type studied by several 
authors (see, e.g. \cite{BO, V}) and they motivated the Jacobi polynomials of 
Heckman and Opdam \cite{HO} associated with root systems.  The connection between
these $BC_2$ polynomials and orthogonal polynomials for $\CW_{\a,\b,\g}$ appears to 
be new. 

The paper is organized as follows. The following section is a preliminary, where the basic
results on orthogonal polynomials are introduced. In Section 3 we recollect properties of 
orthogonal polynomials for $W_{\a,\b,\g}$ and establish a closed form formula for the
reproducing kernel. The orthogonal polynomials for $W_{\a,\b,\g}$ are studied in Section 4. 
The orthogonal expansions are investigated in Section 5. 

\section{Preliminary on orthogonal polynomials}
\setcounter{equation}{0}

In this short section we recall basics on orthogonal polynomials of one variable and two 
variables, respectively, in two separate subsections. 

\subsection{Orthogonal polynomials of one variable}
Let $w$ be a nonnegative weight function on $[-1,1]$ 
that has finite moment of all orders. Throughout this paper we denote by $p_n$ the 
orthonormal polynomials of degree $n$ with respect to the weight function $w$, which 
are uniquely determined by 
$$
    \int_{-1}^1 p_n(x) p_m(x)w(x) dx = \delta_{n,m}, \quad n, m \ge 0
$$
and $\g_n >0$,  where $\g_n$ denotes the leading coefficient of the orthogonal polynomial 
$p_n$, that is, $p_n(x) = \g_n x^n + \ldots$. Let $\Pi_n$ denote the space of polynomials of 
degree at most $n$ in one variable. The reproducing kernel of $k_n(w; \cdot,\cdot)$ of $\Pi_n$ 
is defined by the relation
$$
    \int_{-1}^1 p(y) k_n(w;x,y) w(y) dy = p (x), \quad \forall p \in \Pi_n. 
$$
The well known Christoffel-Darboux formula shows that 
\begin{equation} \label{CDformula}
    k_n(w; x,y) = \sum_{k=0}^n  p_k(x) p_k(y) = 
      \frac{\g_{n}}{\g_{n+1}} \frac{p_{n+1}(x) p_n(y)- p_{n+1}(y) p_n(x)}{x-y}. 
\end{equation}
The Fourier orthogonal expansion of $f \in L^2(w)$ is defined by
$$
  f  = \sum_{n=0}^\infty \wh f_n  p_n \quad \hbox{and} \quad 
      \wh f_n : = \int_{-1}^1 f(y) p_n (y) w (y)dy, 
$$
where the equality holds in the $L^2(w)$ sense by the standard Hilbert space theory and
the fact that polynomials are dense in $L^2(w)$. 
The partial sum operator $s_n f$ of this expansion is given by 
\begin{equation}\label{eq:partial}
    s_n (w; f, x) := \sum_{k=0}^n \wh f_k p_k  
       = \int_{-1}^1 f(y) k_n (w; x,y)  w(y)dy,
\end{equation}
where the second equal sign follows from the definition of $k_n(w;\cdot,\cdot)$.

The Jacobi weight function $w = w_{\a,\b}$ ($w \in J$) is defined by 
$$
  w_{\a,\b}(x): = (1-x)^\a (1+x)^\b,  \qquad \a,\b > -1.  
$$
The Jacobi polynomials are orthogonal with respect to $w_{\a,\b}$ and they 
are given explicitly as an hypergeometric function
$$
 P_n^{(\a,\b)}(x,y) = \frac{(\a+1)_n}{n!} {}_2F_1 \left(\begin{matrix} -n, n+\a+\b+1 \\ \a+1 \end{matrix};
        \frac{1-x}{2} \right). 
$$
These polynomials satisfy the orthogonal conditions
\begin{align*}
  c_{\a,\b}  \int_{-1}^1 P_n^{(\a,\b)} (x) P_m^{(\a,\b)} (x) w_{\a,\b}(x) dx 
      =  h_n^{(\a,\b)} \delta_{n,m}, 
\end{align*}
where 
\begin{align*}
       c_{\a,\b}: = \frac{\Gamma(\a+\b+2)}{2^{\a+\b+1} \Gamma(\a+1)\Gamma(\b+1)}, \quad
       h_n^{(\a,\b)}:= \frac{(\a+1)_n (\b+a)_n (\a+\b+n+1)}{n!(\a+ \b+2)_n (\a+\b+ 2n+1)}.
\end{align*}
We denote the orthonormal Jacobi polynomials by $p_n^{(\a,\b)}$. It follows readily that 
$p_n^{(\a,\b)}(x) = (h_n^{(\a,\b)} )^{-\frac12} P_n^{(\a,\b)}(x)$. Furthermore, for $w = w_{\a,\b}$, 
we also write the reproducing kernel as $k_n^{(\a,\b)}(\cdot,\cdot)$ and the partial sum 
operator as $s_n^{(\a,\b)} (f)$. 

More generally, a function $w$ is called a generalized Jacobi weight function ($w \in GJ$) if
it is of the form
\begin{equation} \label{GJ}
   w(x) = \psi(x) (1-x)^{\g_0} (1+x)^{\g_{r+1}} \prod_{i=1}^r |x - x_i|^{\gamma_i},  \quad  \g_i > -1,  
\end{equation}
if $|x| \le 1$ and $w(x) = 0$ if $|x| >1$, where $-1 < x_1 < \ldots < x_r <1$ and $\psi$ is a 
positive continuous function in $[-1,1]$ and the modulus of continuity $\omega$ of $\psi$ satisfies
$$
    \int_0^1 \frac{\omega(t)}{t} dt < \infty,
$$
which holds, in particular, if $\psi$ is continuously differentiable. For a class $GJ$, the points 
$x_1,\ldots, x_r$ are fixed whereas $\g_1,\ldots, \g_r$ are parameters. In the case of 
$\g_i = 0$, $1 \le i \le r$ and $\psi(x) = 1$, $w$ is an ordinary Jacobi weight function. Orthogonal polynomials
with respect to $w \in GJ$ are called generalized Jacobi polynomials. They share many 
properties of Jacobi polynomials (see, e.x., \cite{Ba, N}), even though 
they do not have explicit formulas in terms of hypergeometric functions.  

\subsection{Orthogonal polynomials of two variables} Let $W$ be a nonnegative weight 
function defined on a bounded domain $\Omega \subset \RR^2$. We define an inner product
\begin{equation} \label{eq:inner}
    \la f, g \ra_W: = \int_\Omega f(x_1,x_2) g(x_1,x_2) W(x_1,x_2) dx_1dx_2 
\end{equation}
on the space of polynomials. Let $\Pi_n^2$ denote the space of polynomials of (total) degree 
at most $n$ in two variables. A polynomial $P \in \Pi_n^2$ is called orthogonal if $\la P, Q\ra = 0$ for 
all $Q \in \Pi_{n-1}^2$. Let $\CV_n(W)$ denote the space of such orthogonal polynomials 
of degree $n$. Then 
$$
       \dim \CV_n(W) = n+1, \qquad \dim \Pi_n^2 = \binom{n+2}{2}. 
$$

The space $\CV_n(W)$ can have many different bases. We usually index the elements 
of a basis by $\{P_{k,n}: 0 \le k \le n\}$. A basis of 
$\CV_n(W)$ is called mutually orthogonal if 
$$
    \la P_{k,n}, P_{j,n} \ra_W = h_k \delta_{k,j},  \qquad 0 \le k,j \le n,
$$
and it is called orthonormal if $h_k =1$, $0 \le k \le n$.
The reproducing kernel $K_n(W; \cdot,\cdot)$ of $\Pi_n^2$ in $L^2(W)$ is defined uniquely by 
$$
    \int_\Omega K_n(W; x,y) f (y) W(y) dy = f(x),     \quad \forall f \in \Pi_n^2, 
$$
where $x = (x_1,x_2)$ and $y=(y_1,y_2)$. Let $\{P_{k,m}: 0 \le k \le m\}$ be a sequence 
of orthonormal polynomials with respect to $W$. Then the kernel 
$K_n(W;\cdot,\cdot)$ satisfies 
\begin{equation} \label{reprodK}
   K_n(W; x,y) = \sum_{m=0}^n \sum_{k=0}^m P_{k,m}(x) P_{k,m}(y). 
\end{equation}
Since $W$ is bounded, polynomials are dense in $L^2(W)$. For $f \in L^2(W)$, the orthogonal 
expansion of $f$ is defined by  
$$
  f = \sum_{n=0}^\infty \sum_{k=0}^n \wh f_{k,n} P_{k,n}, \quad\hbox{where} \quad
   \wh f_{k,n} = \int_\Omega f(y) P_{k,n}(y) W(y) dy. 
$$
The $n$-th partial sum operator of the above expansion is give by
\begin{equation} \label{partialS}
 S_n(W; f) :=  \sum_{k=0}^n \sum_{j=0}^k \wh f_{j,k} P_{j,k} = \int_{\Omega} f(y) K_n(W; \cdot, y)W(y) dy, 
\end{equation}
where the second equal sign follows from \eqref{reprodK}. 
There is an analogue of the Christoffel-Darboux formula for this kernel (\cite[p. 109]{DX})
but it still involves a summation and is not as useful. For studying convergence of the 
orthogonal expansions beyond $L^2$, it is often necessary to have a compact formula for
the kernel. 

\begin{rem}
The partial sum $S_n(W;f)$ in \eqref{partialS} is defined in terms of the polynomial space $\Pi_n^2$ in total
degree. Since $p_n(x) p_m(x)$ has degree $n+m$, the partial sum for the product weight function, say  
$W_{\a,\b} (x,y) = w_{\a,\b}(x)w_{\a,\b}(y)$, does not have a product structure. In fact, there is no compact
formula for the kernel $K_n(x,y)$ for $W_{\a,\b}$ if $(\a,\b) \ne (-\frac12, - \frac12)$. As a consequence, 
there is little progress on the study of orthogonal expansions on the square.
\end{rem}

\section{Koornwinder orthogonal polynomials}
\setcounter{equation}{0}

The definition and the properties of the Koornwinder orthogonal polynomials are discussed in
the first subsection. A new compact formula for the reproducing kernels is given in the second 
subsection.

\subsection{Orthogonal polynomials}
Let $w$ be a nonnegative weight function defined on $[-1,1]$.  For $\g > -1$ define
\begin{equation}\label{BCweight}
          B_\gamma (x,y) := a_w^\g w(x) w(y) |x-y|^{2\g +1}, \qquad (x,y) \in [-1,1]^2, 
\end{equation}
where $a_w^\g$ is a normalization constant such that $\int_{[-1,1]^2} B_\g(x,y)dxdy  =1$. 

Since $B_\g$ is evidently symmetric in $x,y$, we only need to consider its restriction on 
the triangular domain $\Delta$ defined by 
$
      \triangle : = \{(x,y): -1< x< y<1\}. 
$
Let $\Omega$ be the image of $\Delta$ under the mapping $(x,y) \mapsto (u,v)$ defined by 
\begin{equation}   \label{u-v} 
        u = x+y, \quad v= xy.
\end{equation}
It is easy to see that this mapping is a bijection between $\triangle$ and $\Omega$. 
The domain $\Omega$ is given by 
$$
       \Omega: = \{(u,v): 1+u+v > 0, 1-u+v > 0, u^2 > 4v\}
$$
and it is depicted in Figure 1. 
\begin{figure}[h]
\includegraphics[scale=0.48]{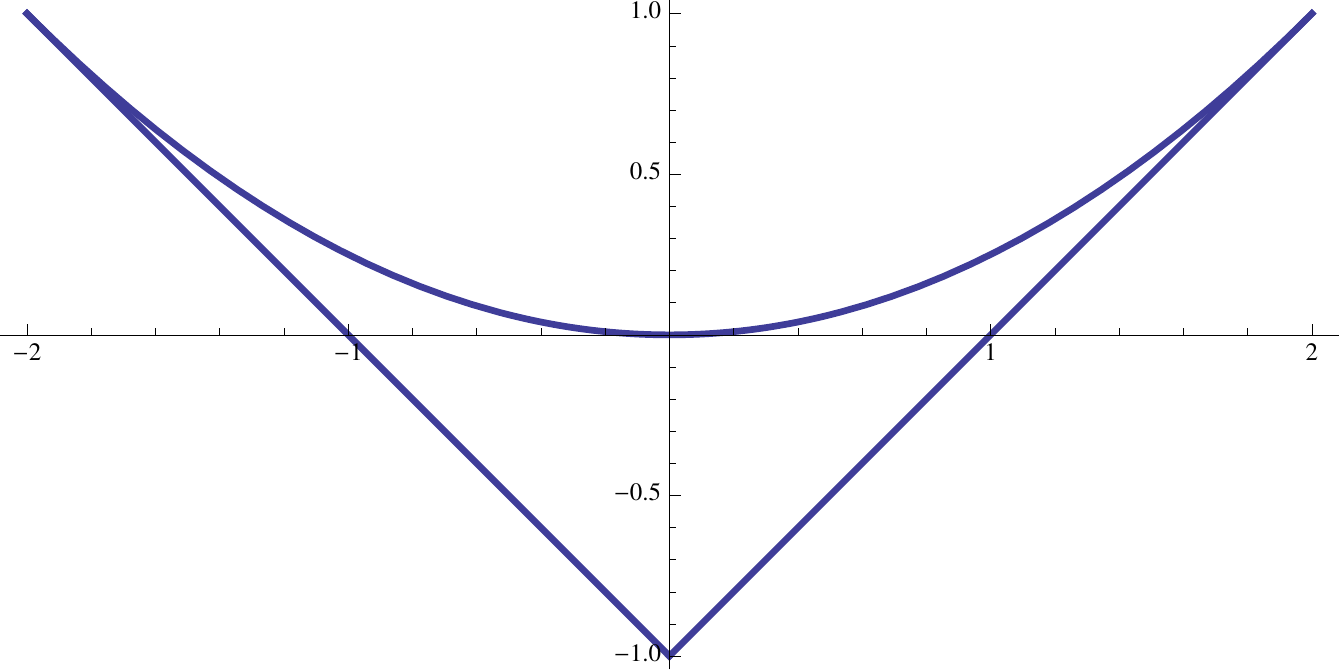}
\caption[Region of Koornwinder' orthogonal polynomials]{Domain $\Omega$}
\label{figure:region} 
\end{figure} 
This is exactly the domain defined in \eqref{Omega}. 

We consider a family of weight functions defined on the domain $\Omega$ by 
\begin{equation} \label{Wgamma}
      W_\g(u,v)  = 2 a_w^\g w(x) w(y) (u^2 - 4 v)^\g, \qquad (u,v) \in \Omega, 
\end{equation}
where the variables $(x,y)$ and $(u,v)$ are related by \eqref{u-v}. The Jacobian of the 
change of variables \eqref{u-v} is given by $du dv = |x-y|dxdy$. Moreover, $u^2 - 4 v = 
(x-y)^2$.  It follows that 
\begin{align} \label{Para-Square}
  \int_{\Omega} f(u,v) W_\g(u,v) dudv & = 
       2 \int_{\triangle}  f(x+y, x y) B_\g(x,y) dxdy \\
       & =  \int_{[-1,1]^2}  f(x+y, x y) B_\g(x,y) dxdy, \notag
\end{align}
where the second equal sign follows since the integrant is a symmetric function of 
$x$ and $y$, and $[-1,1]^2$ is the union of $\triangle$ and its image under 
$(x,y) \mapsto (y,x)$.  In particular, setting $f(x,y) =1$ shows that $W_\g$ is a 
normalized weight function.  

In the case of $w(x) = w_{\a,\b}(x)$, the weight function $W_\gamma$ becomes
$W_{\a,\b,\g}$ in \eqref{Wabg}, which we restate below,
\begin{equation} \label{Wabc}
    W_{\a,\b,\g}(u,v) : = 2 a_{\a,\b,\g}(1-u+v)^\a (1+u+v)^\b (u^2- 4 v)^\g, \quad (u,v) \in \Omega,
\end{equation}
where the constant $a_{\a,\b,\g}$ is given by \cite[Lemma 6.1]{S}, 
\begin{align}
 a_{\a,\b,\g}: = \frac{\sqrt{\pi} }{2^{2\a+2\b+4\g+4}} \frac{\Gamma(\a+\b+\g+\frac52)\Gamma(\a+\b+2 \g+3)}
    {\Gamma(\a+1) \Gamma(\b+1)\Gamma(\g+1) \Gamma(\a+\g + \frac32)
      \Gamma(\b+\g+\frac32)},
\end{align}
which is the two-variable case of the Selberg integral (e.x., (1.1) of \cite{FW}). 
This weight function is integrable on $\Omega$ if $\a,\b,\g > -1$ and $\a + \g +3/2 >0$
and $\b + \g + 3/2 > 0$, and we assume that $\a,\b,\g$ satisfy these inequalities from 
now on. Other examples of $W_\g$ include 
$$
    e^{- a u}  (u^2 - 4 v)^\g \quad \hbox{and}\quad e^{- a  (u^2 - 2 v)}  (u^2 - 4 v)^\g,  \quad a > 0.
$$
which correspond to the choices of $w(x)  = e^{- a x}$ and $w(x) = e^{- a x^2}$. 

Let ${\mathcal N} = \{(k,n): 0 \le k \le n\}$. In ${ \mathcal N}$ define $(j,m) \prec (k,n)$ 
if $m < n$ or $j \le k$ when $m = n$. Then the orthogonal polynomials $P_{k,n}^{(\g)}$ 
that satisfy 
\begin{equation}\label{eq:jacobi-biangle}
   P_{k,n}^{(\g)}(u,v) = u^{n-k} v^k + \sum_{(j,m) \prec (k,n)} a_{j,m} u^{m-j} v^j 
\end{equation}
and the orthogonality condition 
\begin{equation}\label{eq:biangle-ortho}
     \int_\Omega P_{k,n}^{(\g)}(u,v) u^{m-j} v^j W_{\g}(u,v) dudv =0, 
           \quad \forall (j,m) \prec (k,n), 
\end{equation}
are uniquely determined, as can be seen by the Gram-Schmidt process. The polynomials 
$P_{k,n}^{(\g)}$ are mutually orthogonal. 

When $W_\g = W_{\a,\b,\g}$, we denote these 
orthogonal polynomials by $P_{k,n}^{\a,\b,\g}$. 

In the case of $\g = \pm  \frac12$, these orthogonal polynomials can be given explicitly,
as can be easily verified upon using \eqref{Para-Square}. 
Let $p_n$ denote the orthogonal polynomial of degree $n$ with respect to $w$.  
Then an orthonormal basis with respect to $W_{-\frac12}$ is given by 
\begin{equation} \label{OP-1/2}
    P_{k,n}^{(-\frac12)} (u,v) =  \frac{1}{\sqrt{2}}\begin{cases} 
               p_n(x) p_k(y) + p_n(y) p_k(x), & 0 \le k < n, \\ 
               \sqrt{2} p_n(x) p_n(y), & k =n,\end{cases} 
\end{equation}
and an orthonormal basis with respect to $W_{\frac12}$ is given by 
\begin{equation} \label{OP+1/2}
    P_{k,n}^{(\frac12)} (u,v) = \sqrt{\frac{a_w^{- 1/2}}{2 a_w^{1/2}}}\,  \frac{p_{n+1}(x) p_k(y) - p_{n+1} (y) p_k(x)}{x-y}, \quad 0 \le k \le n, 
\end{equation}
both families are defined under the mapping \eqref{u-v}. It should be noted that the polynomials 
in \eqref{OP-1/2} and \eqref{OP+1/2} are normalized by their orthonormality, instead of 
by the leading coefficient as in \eqref{eq:biangle-ortho}, but they have the same structure
as those in \eqref{eq:biangle-ortho}, so that the difference is just a constant multiple.  
In the case of $w=w_{\a,\b}$,  the orthogonal polynomials in \eqref{OP-1/2} and \eqref{OP+1/2} 
are denoted by $P_{k,n}^{\a,\b,-\frac12}$ and $P_{k,n}^{\a,\b,\frac12}$, respectively, and
they are expressed in terms of Jacobi polynomials $p_n^{(\a,\b)}$. 

For $W_{\a,\b,\g}$ these orthogonal polynomials were first studied by Koornwinder in 
\cite{K74}, see also \cite{K75}. The above statements for more general weight function $W_\g$ 
are straightforward extensions and used in \cite{SX} for studying Gaussian cubature rules. Much 
more can be said about the orthogonal polynomials $P_{k,n}^{\a,\b,\g}$. They are, for example, 
eigenfunctions of two differential operators of order 2 and order 4, respectively \cite{K74}. Another 
pair of differential operators were constructed in \cite{S},  
\begin{align} \label{Eab}
\begin{split}
& E_-^{\a,\b}: = u \frac{\partial^2}{\partial u^2} + 2(v + 1) \frac{\partial^2}{\partial u \partial v} +
    u \frac{\partial^2}{\partial v^2} + (\b - \a)\frac{\partial}{\partial v} + (\a + \b +2) \frac{\partial}{\partial u}\\
& E_+^{\a,\b,\g}: = \left[W_{\a,\b,\g}(u,v)\right]^{-1} E_-^{-\a, -\b} W_{\a,\b,\g+1}(u,v),
\end{split}
\end{align}
and they act as raising and lowering operators on the orthogonal polynomials,
\begin{align*}
   E_-^{\a,\b}P_{k,n}^{\a,\b,\g} (u,v) & =  (n-k)(n+k+\a+\b+1)P_{k,n-1}^{\a,\b,\g+1}(u,v), \\
   E_+^{\a,\b,\g}P_{k,n-1}^{\a,\b,\g+1} (u,v) & =  (n-k+2 \g+1)(n+k+\a+\b+2 \g+2)P_{k,n}^{\a,\b,\g}(u,v)
\end{align*}
for $0 \le k \le n-1$ and $E_-^{\a,\b}P_{n,n}^{\a,\b,\g} (u,v) =0$. Together these two operators
can be used to give a Rodrigues type formula for $P_{k,n}^{\a,\b,\g}$ (\cite[(5.1)]{S}) and they
can also be used to calculate the $L^2$-norms of $P_{k,n}^{\a,\b,\g}$ and the coefficients in the
recurrence relation.   

The polynomials $P_{k,n}^{\a,\b,\g}(u,v)$ also satisfy a quadratic transformation formula 
\cite[Theorem 10.1]{S} given by, for $0 \le k \le n$,
\begin{align} \label{quadraticPkn}
\begin{split}
   P_{n-k,n+k}^{\a,\a,\g}(u,v) & = 2^{-n+k} P_{k,n}^{\g, - \frac12, \a}(2 v, u^2 - 2 v -1), \\
  u^{-1} P_{n-k,n+k+1}^{\a,\a,\g}(u,v) & = 2^{-n+k} P_{k,n}^{\g, \frac12, \a}(2 v, u^2 - 2 v -1). 
\end{split} 
\end{align}
In particular, setting $\g = \pm \frac12$ and $\a \mapsto \g$ and let $s = 2 xy$ and $t = x^2+y^2 -1$, 
it follows that \cite[p. 518]{S},
\begin{align} \label{P1/21/2g}
\begin{split}
P_{k,n}^{-\frac12, -\frac12, \g}(s,t) & =  
   c ( p_{n+k}^{\g,\g}(x) p_{n-k}^{\g,\g}(y) + p_{n-k}^{\g,\g}(x) p_{n+k}^{\g,\g}(y) ), \\
P_{k,n}^{\frac12, -\frac12, \g}(s,t) & =  c (x-y)^{-1} \left[ 
     p_{n+k+1}^{\g,\g}(x) p_{n-k}^{\g,\g}(y) - p_{n-k}^{\g,\g}(x) p_{n+k+1}^{\g,\g}(x) \right], \\   
P_{k,n}^{- \frac12, \frac12, \g}(s,t) & =c  (x+y)^{-1} \left[ 
     p_{n+k+1}^{\g,\g}(x) p_{n-k}^{\g,\g}(y) + p_{n-k}^{\g,\g}(x) p_{n+k+1}^{\g,\g}(x) \right], \\    
P_{k,n}^{\frac12, \frac12, \g}(s,t) & = c (x^2 - y^2)^{-1} \left[ 
     p_{n+k+2}^{\g,\g}(x) p_{n-k}^{\g,\g}(y) - p_{n-k}^{\g,\g}(x) p_{n+k+2}^{\g,\g}(x) \right],    
\end{split}
\end{align}
where $c$ is a constant proportional to $2^{n-k}$. In other words, a basis of orthogonal polynomials for 
$W_{ \pm \frac12,  \pm \frac12, \g}$ can be explicitly given in terms of Jacobi 
polynomials. 

For further results on $P_{k,n}^{\a,\b,\g}$, including explicit series expansions and recursive 
relations, see \cite{K74, KS, S}. 

It is worth to mention that the relation \eqref{Para-Square} shows that orthogonal polynomials
$P_{k,n}$ for $W_\g$ are closely related to the orthogonal polynomials with respect to $B_\g$ 
on $[-1,1]^2$, as seen by
$$
        W_\g(x+y,xy) = B_\g(x,y) |x-y|
$$
and \eqref{Para-Square}. Indeed, if $R_{k,n}$ is polynomial orthogonal to $x^l y^m$, if $\max \{l,m\}<n$,
with respect to $B_\g$ on $[-1,1]^2$, then $R_{k,n}(x,y)+ R_{k,n}(y,x)$ is a symmetric polynomial
orthogonal to  $x^l y^m + x^m y^l$, if $\max \{l,m\}<n$, for $B_\g$ since $B_\g$ is symmetric. 
Hence, under the bijection \eqref{u-v}, the polynomial
$$
     P_{k,n}(u,v) := R_{k,n}(x,y) + R_{k,n}(y,x)
$$
is an orthogonal polynomial with respect to $W_\g$ on $\Omega$. Since the mapping 
\eqref{u-v} is not linear, one needs to be careful about the degree of $P_{k,n}$. In the case of
$w = w_{\a,\b}$, the symmetric orthogonal polynomials for $B_\g$ are the $BC_2$ type 
polynomials, the precursor of the generalized Jacobi polynomials of $BC_n$ type.  

\subsection{Reproducing kernel}
Recall that $K_n(W_\g; \cdot, \cdot)$ denotes the reproducing kernel of $\Pi_n^2$ in $L^2(W_\g)$,
which we shall denote by $K_n^{(\g)}(\cdot,\cdot)$ below. In contrast to \eqref{reprodK}, we derive a
closed formula for $K_n^{(\g)}(\cdot,\cdot)$ in the case of $\g = \pm \frac12$ in this subsection. 

\begin{thm} \label{thm:reprodW}
Let $k_n(\cdot,\cdot):= k_n(w;\cdot, \cdot)$ be the kernel defined in \eqref{CDformula}. Set
$$
  u := (u_1,u_2) = (x_1+x_2, x_1 x_2) \quad  \hbox{and}\quad  v:= (v_1,v_2) = ( y_1+y_2, y_1 y_2).
$$
Then the reproducing kernel $K_n^{(-\frac12)}(\cdot, \cdot)$ for $W_{- \frac12}$ is given by 
\begin{equation} \label{eq:reprodW-1/2}
    K_n^{(-\frac12)}(u,v) =  \frac{1}{2} \left [ k_n(x_1,y_1)k_n(x_2,y_2)+ k_n(x_2,y_1) k_n(x_1,y_2) \right], 
\end{equation}
and the reproducing kernel $K_n^{(\frac12)}(\cdot, \cdot)$ for $W_{\frac12}$ is given by 
\begin{align} \label{eq:reprodW+1/2}
 K_n^{(\frac12)}(u,v) =\frac12  \frac{k_{n+1}(x_1,y_1)k_{n+1}(x_2,y_2) - 
   k_{n+1}(x_2,y_1) k_{n+1}(x_1,y_2)} {(x_1-x_2)(y_1-y_2)}. 
\end{align}
\end{thm}

\begin{proof}
Denote the right hand side of \eqref{eq:reprodW-1/2} by $\wh k_n( (x_1,x_2), (y_1,y_2))$. 
By the definition of $k_n(\cdot,\cdot)$ in \eqref{CDformula}, for a fixed $y_1,y_2$, we have
$$
  \wh k_n( (x_1,x_2), (y_1,y_2)) = \frac12 \sum_{k=0}^n \sum_{j=0}^n p_k(y_1)p_j(y_1)
     \left[ p_k(x_1)p_j(x_2) + p_k(x_2)p_j(x_1)\right],
$$
which shows, upon setting $u_1 = x_1+x_2$, $u_2 = x_1x_2$, that $\wh k_n( (x_1,x_2), (y_1,y_2))$
is a polynomial of degree $n$ in $(u_1,y_2)$. Hence, if we define $\wh K_n((u_1,u_2), (v_1,v_2)) = 
\wh k_n( (x_1,x_2), (y_1,y_2))$ under the mapping $(u_1,u_2) \mapsto (x_1+x_2,x_1 x_2)$ 
and $(v_1,v_2) \mapsto (y_1+y_2, y_1y_2)$, then $\wh K$ is a polynomial of degree $n$ in $(u_1,u_2)$
and, by symmetry, in $(v_1,v_2)$. Thus, we only have to verify the reproducing property. 
For $0 \le j \le m \le n$, the reproducing property of $k_n(\cdot,\cdot)$ implies immediately 
\begin{align*}
 & \int_{\Omega}  \wh K_n( (u_1,u_2), (v_1,v_2)) P_{j,m}^{(- \frac12)}(v_1,v_2)
          W_{-\frac12}(v_1,v_2)dv_1dv_2 \\
          =  &  \int_{[-1,1]^2}  \wh k_n( (x_1,x_2), (y_1,y_2)) P_{j,m}^{(- \frac12)}(y_1+y_2, y_1y_2) 
          w(y_1)w(y_2) dy_1dy_2 \\
          = & \, \frac{1}{\sqrt{2}} \left[ p_m(x_1)  p_j(x_2) +  p_j(x_1) p_m(x_2)\right]  = P_{j,m}^{(- \frac12)}(u_1,u_2),
\end{align*}
which shows that $\wh K_n( (u_1,u_2), (v_1,v_2))$ is the reproducing kernel of $\Pi_n^2$, so
that \eqref{eq:reprodW-1/2} holds. 

Denote now the right hand side of \eqref{eq:reprodW+1/2} by $\wh k_n( (x_1,x_2), (y_1,y_2))$. 
Then, for fixed $(y_1, y_2)$, it is easy to see that 
$$
  \wh k_n( (x_1,x_2), (y_1,y_2)) = \sum_{j=0}^{n+1} \sum_{k=0}^{n+1} 
     \frac{p_j(y_1)p_k(y_2)}{y_1-y_2} \frac{ p_j(x_1) p_k(x_2) -p_j(x_2) p_k(x_1)}{x_1-x_2}, 
$$
which shows, since the terms for $k = j$ are zero, that $\wh k_n( (x_1,x_2),(y_1,y_2))$ 
is a polynomial of degree $n$ in $(u_1,u_2)$, where $u_1 = x_1+x_2$ and $u_2 = x_1 x_2$. 
By symmetry, the same holds for $(y_1,y_2)$ with $(x_1,x_2)$ fixed. Thus, it remains to 
prove the reproducing property, which works similarly as in the case of $\g = -1/2$ upon 
using the fact that $(y_1-y_2)^2$ in $W_{\frac12}(u_1,u_2) =  (y_1-y_2)^2 w(y_1)w(y_2)$ 
cancels the denominators in both $\wh k_n(\cdot,\cdot)$ and $P_{j,m}^{\frac12} (y_1+y_2,y_1y_2)$. 
\end{proof}

The closed formula for the reproducing kernel allows us to study the convergence of the Fourier
orthogonal expansions, which will be discussed in Section 5. 
 
\section{Orthogonal polynomials for weight functions on $[-1,1]^2$} 
\setcounter{equation}{0}

In this section we study orthogonal polynomials for the family of weight functions defined 
in \eqref{Wgamma} on $[-1,1]^2$, which are closely related to $W_\g$. The orthogonal 
polynomials are given in the first subsection, their further properties are in the second 
subsection, and a compact formula for the reproducing kernel is in the third subsection. 

\subsection{Orthogonal polynomials}
Let $w$ be the weight function on $[-1,1]$ and let $W_\g$ be the weight function defined 
in \eqref{Wgamma}. We then define 
\begin{equation} \label{CWgamma}
        \CW_\g (x,y) := W_\g (2 x y, x^2+y^2 -1)|x^2-y^2|, \qquad (x,y) \in [-1,1]^2, 
\end{equation}
which is the quadratic transform that has appeared in \eqref{quadraticPkn}. 
Let $\Omega$ be the domain \eqref{Omega} of $W_\g$. The fact that $(x,y) \in [-1,1]^2$ 
implies that $(2xy, x^2+y^2 -1) \in \Omega$ is shown in the lemma below. In the case 
that $W_\g = W_{\a,\b,\g}$, the weight function $\CW_\g$ becomes, up to a constant, 
$\CW_{\a,\b,\g}$ defined in \eqref{CWabg}, which we restate as, 
\begin{align} \label{CWabc}
   \CW_{\a,\b, \g}(x,y) : = 2 a_{\a,\b,\g}4^\g |x-y|^{2\a +1}  |x + y|^{2\b +1}   (1-x^2)^\g (1-y^2)^\g, 
\end{align}
where $\a, \b, \g > -1$, $\a + \g + \frac12> -1$ and $\b+ \g + \frac12> -1$. These conditions
on the parameters, ensuring the integrability of $\CW_{\a,\b,\g}$, are the same as those for
$W_{\a,\b,\g}$. The weight function $\CW_\g$, and $\CW_{\a,\b,\g}$, is normalized as 
shown below. We define a region $\Omega^* : = \{ (x,y): - 1  < x < y <1\}$.

\begin{lem} \label{Int-Para-cube}
The mapping $(x,y)  \mapsto (2 x y , x^2+y^2 -1)$ is a bijection from $\Omega^*$ onto 
$\Omega$. Furthermore,
\begin{align} \label{Int-P-Q}
 \int_{\Omega} f(u,v) W_{\g} (u,v) du dv  = \int_{[-1,1]^2} f(2xy, x^2+y^2 -1) \CW_{\g}(x,y) dx dy.
\end{align}
\end{lem}

\begin{proof}
For $(x,y) \in [-1,1]^2$, let us write $x = \cos \t$ and $y = \cos \phi$, $0 \le \t, \phi \le \pi$. Then 
it is easy to see that
\begin{equation} \label{theta-phi}
    2 x y = \cos (\t - \phi) + \cos (\t + \phi), \quad x^2+y^2 -1 = \cos (\t -  \phi) \cos (\t + \phi), 
\end{equation}
from which it follows readily that $(2xy, x^2+y^2 -1) \in \Omega$. For the change of 
variable $u= 2xy$ and $v = x^2+y^2 -1$, we have $du dv = 4 |x^2-y^2| dxdy$, 
from which the stated formula follows. 
\end{proof}

As a more general example, we consider $w$ being a generalized Jacobi weight.

\begin{lem}\label{lem:GJ}
Let $w \in GJ$ be defined as in \eqref{GJ}. Then
\begin{align} \label{GJ-exp}
   \CW_{\g}(x,y) =  \Psi(x,y) |x-y|^{2 \g_0+1} & \prod_{k=1}^r \left |1+t_k^2 - (x - t_k)^2 - (y-t_k)^2 \right|^{ \g_k} \\
         & \qquad \times  |x+y|^{2 \g_{r+1}+1} (1-x^2)^\g (1-y^2)^\g, \notag
\end{align}
where $\Psi(\cos \t, \cos\phi) = \psi(\cos(\t-\phi)) \psi(\cos(\t+\phi))$. 
\end{lem}

We consider orthogonal polynomials with respect to the inner product 
\begin{align} \label{IpCW}
  \la f, g \ra_{\CW_\g} :=  \int_{[-1,1]^2} f(x,y) g(x,y) \CW_\g (x,y) dxdy. 
\end{align}
Let $\CV_{n}(\CW_{\g})$ denote the space of orthogonal polynomials of degree $n$ 
with respect to the inner product $\la \cdot,\cdot \ra_{\CW_\g}$. It turns out that a basis 
of $\CV_{n}(\CW_{\g})$ can be expressed in terms of orthogonal polynomials with 
respect to $W_{\g}$ and three other related weight functions. Recall that $\la \cdot,\cdot\ra_W$
is defined in \eqref{eq:inner}. We define three other weight  functions 
\begin{align}\label{Wij}
\begin{split}
  W_\g^{(1,1)}(u,v) &  :=  2 a_{w^{(1,1)}}^\g (1-u+v)(1+u+v)W_\g(u,v), \\
  W_\g^{(1,0)}(u,v)  & := 2 a_{w^{(1,0)}}^\g (1-u+v)W_\g(u,v), \\
  W_\g^{(0,1)}(u,v)  & := 2  a_{w^{(0,1)}}^\g (1+u+v)W_\g(u,v), 
\end{split} \quad (u,v) \in \Omega,
\end{align} 
where we define 
\begin{equation} \label{eq:wij}
          w^{(i,j)}(x) := (1-x)^i (1+x)^j w(x), \quad i,j = 0,1.
\end{equation}
Under the change of variables $u = x+y$ and $v = xy$, $1-u+v = (1-x)(1-y)$ and 
$1+u+v = (1+x)(1+y)$. The three weight functions in \eqref{Wij} are normalized 
so that $\int_\Omega W_\g^{i,j}(u,v)dudv = 1$. Clearly these three weight functions
are of the same type as $W_\g$. 

In the following theorem, we denote by $\{P_{k,n}^{(\g)}: 0 \le k \le n\}$ an orthonormal
basis of $\CV_n(W_{\g})$ under $ \la \cdot, \cdot \ra_{W_{\g}}$. For $0 \le k \le n$, we 
further denote by $P_{k,n}^{(\g),1,1}$, $P_{k,n}^{(\g),1,0,}$, $P_{k,n}^{(\g),0,1}$ the 
orthonormal polynomials of degree $n$ with respect to $\la f, g \ra_W$ for $W = W_\g^{(1,1)}$, 
$W_\g^{(1,0)}$, $W_\g^{(0,1)}$, respectively. 
 
\begin{thm} \label{thm:Qop}
For $n = 0,1,\ldots$, an orthonormal basis of $\CV_{2n}(\CW_{\g})$ is given by
\begin{align} \label{Qeven2}
\begin{split}
         {}_1Q_{k,2n}^{(\g)}(x,y):= & P_{k,n}^{(\g)}(2xy, x^2+y^2 -1), \qquad 0 \le k \le n, \\
         {}_2Q_{k,2n}^{(\g)}(x,y) := & b_\g^{(1,1)}(x^2-y^2)  P_{k,n-1}^{(\g),1,1}(2xy, x^2+y^2 -1),  \quad 0 \le k \le n-1, 
 \end{split}
\end{align}
and an orthonormal basis of $\CV_{2n+1}(\CW_{\g})$ is given by 
\begin{align} \label{Qodd2}
 \begin{split}
     &  {}_1Q_{k,2n+1}^{(\g)}(x,y):= b_\g^{(0,1)}(x+ y) P_{k,n}^{(\g),0,1}(2xy, x^2+y^2 -1), \qquad  0 \le k \le n,  \\
     &  {}_2Q_{k,2n+1}^{(\g)}(x,y):= b_\g^{(1,0)}(x-y) P_{k,n}^{(\g),1,0}(2xy, x^2+y^2 -1), \qquad 0 \le k \le n, 
 \end{split}
\end{align}
where $b_\g^{(i,j)}:= \sqrt{a^\g_{w^{(i,j)}} /a^\g_w}$ for $i,j = 0,1$. 
\end{thm}
 
\begin{proof}
These polynomials evidently form a basis if they are orthogonal.  By Lemma \ref{Int-Para-cube}, 
for $0 \le j \le m$ and $0\le k \le n$,
$$
     \left \langle  {}_1Q_{k,2n}^{(\g)},  {}_1Q_{j,2m}^{(\g)}   \right \rangle_{\CW_{\g}} =  
         \la P_{k,n}^{(\g)}, P_{j,m}^{(\g)} \ra_{W_\g} = \delta_{k,j} \delta_{n,m}, 
$$
and, setting $F = P_{k,n-1}^{(\g), 1,1}P_{k,n}^{(\g)}$, 
$$
 \left  \langle  {}_1Q_{k,2n}^{(\g)},  {}_2Q_{j,2m}^{(\g)} \right \rangle_{\CW_{\g}}
     =   \int_{[-1,1]^2} (x^2-y^2) F(2xy,x^2+y^2 -1) \CW_{\g}(x,y) dxdy.
$$
The right hand side of the above equation changes sign under the change of variables 
$(x,y) \mapsto (y,x)$, which shows that 
$ \left  \langle  {}_1Q_{k,2n}^{(\g)},  {}_2Q_{j,2m}^{(\g)}\right \rangle_{\CW_{\g}} =0$. 
Moreover, since $(x^2-y^2)^2 \CW_{\g}(x,y)dxdy$ is equal to a constant multiple of 
$W_{\g}^{(1,1)}(u,v) dudv$, we see that  
$$
 \left  \langle  {}_2Q_{k,2n}^{(\g)},  {}_2Q_{j,2m}^{(\g)}\right \rangle_{\CW_{\g}}
  = \left  \langle P_{k,n-1}^{(\g),1,1}, P_{j,m-1}^{(\g), 1,1} \right \rangle_{W_\g^{(1,1)}}  
      = \delta_{k,j}\delta_{n,m}.
$$
Furthermore, setting $F = P_{k,n}^{(\g)}P_{k,n}^{(\g), 0,1}$, we obtain 
$$
   \left  \langle  {}_1Q_{k,2n}^{(\g)},  {}_1Q_{j,2m+1}^{(\g)}\right \rangle_{\CW_{\g}} = 
      \int_{[-1,1]^2} (x+y) G(2xy,x^2+y^2 -1) \CW_{\g}(x,y) dxdy,
$$
which is equal to zero since the right hand side changes sign under $(x,y) \mapsto (-x,-y)$. 
The same proof shows also $ \left  \langle  {}_1Q_{k,2n}^{(\g)}, 
{}_2Q_{j,2m+1}^{(\g)}\right \rangle_{\CW_{\g}} =0$.
Together, we have proved the orthogonality of ${}_1Q_{k,2n}^{(\g)}$ and 
${}_2Q_{k,2n}^{(\g)}$. 

Since $(x-y)(x+y) = x^2 - y^2$ changes sign under $(x,y) \mapsto (y,x)$, the same 
consideration shows that $ \left  \langle  {}_1Q_{k,2n+1}^{(\g)}, 
 {}_2Q_{j,2m+1}^{(\g)}\right \rangle_{\CW_{\g}} = 0$.  Finally, we also have 
\begin{align*}
&   \left  \langle  {}_1Q_{k,2n+1}^{(\g)},  {}_1Q_{j,2m+1}^{(\g)}\right \rangle_{\CW_{\g}} =
    \left  \langle P_{k,n}^{(\g), 0,1}, P_{j,m}^{(\g),0,1} \right \rangle_{W_{\g}^{(0,1)} }  
        = \delta_{k,j}\delta_{n,m} \\
&   \left  \langle  {}_2Q_{k,2n+1}^{(\g)},  {}_2Q_{j,2m+1}^{(\g)}\right \rangle_{\CW_{\g}} =
    \left  \langle P_{k,n}^{(\g),1,0}, P_{j,m}^{(\g), 1,0} \right \rangle_{W_{\g}^{(1,0)}} 
        = \delta_{k,j}\delta_{n,m},
\end{align*}
which proves the orthogonality of ${}_1Q_{k,2n+1}^{(\g)}$ and ${}_2Q_{k,2n+1}^{(\g)}$.
\end{proof}

For the weight function $\CW_{\a,\b,\g}$, we denote a basis of orthonormal polynomials 
by $P_{k,n}^{\a,\b,\g}$ in the following theorem. 

\begin{thm} \label{thm:QopJacobi}
For $n = 0,1,\ldots$, an orthonormal basis of $\CV_{2n}(\CW_{\a,\b,\g})$ is given by
\begin{align} \label{Qeven}
\begin{split}
         {}_1Q_{k,2n}^{\a,\b,\g}(x,y):= & P_{k,n}^{\a,\b,\g}(2xy, x^2+y^2 -1), \qquad 0 \le k \le n, \\
         {}_2Q_{k,2n}^{\a,\b,\g}(x,y) := & b_{\a,\b,\g}^{(1,1)}(x^2-y^2)  P_{k,n-1}^{\a+1,\b+1,\g}(2xy, x^2+y^2 -1),  \,\, 0 \le k \le n-1, 
 \end{split}
\end{align}
and an orthonormal basis of $\CV_{2n+1}(\CW_{\a,\b,\g})$ is given by 
\begin{align} \label{Qodd}
 \begin{split}
       &  {}_1Q_{k,2n+1}^{\a,\b,\g}(x,y):=  b_{\a,\b,\g}^{(0,1)}(x+ y) P_{k,n}^{\a,\b+1,\g}(2xy, x^2+y^2 -1), \quad 0 \le k \le n, \\
       &  {}_2Q_{k,2n+1}^{\a,\b,\g}(x,y):=  b_{\a,\b,\g}^{(1,0)}(x-y) P_{k,n}^{\a+1,\b,\g}(2xy, x^2+y^2 -1), \quad  0 \le k \le n,
 \end{split}
\end{align}
where  $b_{\a,\b,\g}^{(i,j)} := \sqrt{a_{\a+i,\b+j,\g} / a_{\a,\b,\g}}$ for $i,j = 0,1$.
\end{thm}

\subsection{Special cases and properties} 
In the case of $\g = \pm \frac12$, we can derive an explicit formula for the basis from
\eqref{OP-1/2} and \eqref{OP+1/2}, which takes a particularly simple form if we change variables
to 
\begin{equation} \label{xy-cos}
   x = \cos \theta, \quad y = \cos \phi, \quad 0 \le \t, \phi \le \pi. 
\end{equation} 

\begin{cor}
For $\g = \pm \frac12$, the orthonormal basis defined in \eqref{Qeven2} and \eqref{Qodd2} of 
$\CV_n(\CW_{\pm \frac12})$ satisfies, under \eqref{xy-cos}, explicit formulas upon using 
the relations 
\begin{align} \label{P-1/2-angle}
 & P_{k,n}^{(- \frac12)} (2 x y, x^2+y^2 -1)  \\ 
 &  \quad =   \frac{1}{\sqrt{2}} \left[p_n (\cos (\t - \phi)) p_k(\cos (\t+\phi)) +  p_k (\cos (\t - \phi)) p_n(\cos (\t+\phi)) \right]\notag
\end{align}
for $0 \le k \le n$, where the $k=n$ term $P_{n,n}^{(- \frac12)}$ is multiplied by the constant $\sqrt{2}$; 
furthermore,  for $0 \le k \le n$, 
\begin{align} \label{P+1/2-angle}
& P_{k,n}^{(\frac12)} (2 x y, x^2+y^2 -1)  \\ 
 &    =  \sqrt{\frac{a_w^{- 1/2}}{2 a_w^{1/2}}}\,  \frac{p_n (\cos (\t - \phi)) p_k (\cos (\t+\phi))  
      -  p_k (\cos (\t - \phi)) p_n (\cos (\t+\phi))}{2 \sin \t \sin \phi}.  \notag
\end{align}
\end{cor}

\begin{proof}
This follows immediately from \eqref{theta-phi}, \eqref{OP-1/2} and \eqref{OP+1/2}. 
\end{proof}

In particular, the orthonormal basis for the weight function 
$$
    \CW_{\a,\b, \pm \frac12} (x_1, x_2)= c |x_1-x_2|^{2 \a +1}  |x_1+x_2|^{2 \b +1} 
         (1-x_1^2)^{\pm \frac12} (1-x_2^2)^{\pm \frac12} 
$$
on $[-1,1]^2$, where $c = 2 a_{a,\b,\pm \frac12} 4^{\pm \frac12}$, can be given in terms of the Jacobi polynomials. 

\begin{prop}
Let $\a, \b > -1$. An orthonormal basis of $\CV_{2n}(\CW_{\a,\b,-\frac12})$ is given 
by, for $0 \le k \le n$ and $0 \le k \le n-1$, respectively,  
\begin{align*} 
    & {}_1Q_{k,2n}^{(\a,\b,-\frac12)}(\cos\t,\cos \phi) = \frac{1}{\sqrt{2}}\left[ p_n^{(\a,\b)} (\cos (\t - \phi)) p_k^{(\a,\b)}(\cos (\t+\phi))\right. \\
    &    \qquad\qquad  \qquad\qquad \qquad\qquad  \qquad 
          \left.   +  p_k^{(\a,\b)} (\cos (\t - \phi)) p_n^{(\a,\b)}(\cos (\t+\phi))\right], \\
    & {}_2Q_{k,2n}^{(\a,\b,-\frac12)}(\cos\t,\cos \phi)  = 
      \frac{  b_{\a,\b,-\frac12}^{(1,1)}} {\sqrt{2}} (x^2-y^2) \left[ p_{n-1}^{(\a+1,\b+1)} (\cos (\t - \phi)) p_k^{(\a+1,\b+1)} (\cos (\t+\phi)) \right. \\ 
    &    \qquad\qquad  \qquad\qquad \qquad\qquad  \qquad 
        \left. +  p_k^{(\a+1,\b+1)} (\cos (\t - \phi)) p_{n-1}^{(\a+1,\b+1)} (\cos (\t+\phi)) \right],
\end{align*}
and an orthonormal basis of $\CV_{2n+1}(\CW_{\a,\b,-\frac12})$ is given by, for $0 \le k \le n$, 
 \begin{align*} 
     & {}_1Q_{k,2n+1}^{(\a,\b,-\frac12)}(\cos\t,\cos \phi)  = \frac{b_{\a,\b,-\frac12}^{(0,1)}} {\sqrt{2}}(x+y) 
         \left[ p_n^{(\a,\b+1)} (\cos (\t - \phi)) p_k^{(\a,\b+1)}(\cos (\t+\phi)) \right. \\ 
     &    \qquad\qquad   \qquad\qquad \qquad\qquad\qquad  \left. 
           +  p_k^{(\a,\b+1)} (\cos (\t - \phi)) p_n^{(\a,\b+1)}(\cos (\t+\phi)) \right], \\
     & {}_2Q_{k,2n+1}^{(\a,\b,-\frac12)}(\cos\t,\cos \phi) = \frac{b_{\a,\b,-\frac12}^{(1,0)}} {\sqrt{2}} (x-y) \left[
          p_n^{(\a+1,\b)} (\cos (\t - \phi)) p_k^{(\a+1,\b)}(\cos (\t+\phi))  \right. \\ 
     &    \qquad\qquad  \qquad\qquad   \qquad\qquad \qquad
          \left.    +  p_k^{(\a+1,\b)} (\cos (\t - \phi)) p_n^{(\a+1,\b)}(\cos (\t+\phi)) \right],
\end{align*}
where whenever $k = n$ the polynomial needs to be multiplied by an additional $\sqrt{2}$. 
\end{prop}

Similarly, an orthonormal basis for $\CV_n(\CW_{\a,\b,\frac12})$ can be given explicitly in terms of the
Jacobi polynomials upon using \eqref{P+1/2-angle}. 

By Theorem \ref{thm:Qop}, the orthogonal polynomials for $\CW_{\a,\b,\g}$ are expressed in terms of
orthogonal polynomials for $W_{\a,\b,\g}$, which in turn are expressed in terms of the symmetric
orthogonal polynomials with respect to the weight function 
$$
   B_{\a,\b,\g}(x,y) = (1-x)^\a(1+x)^\b(1-y)^\a(1+y)^\b |x-y|^{2\g+1}
$$
on $[-1,1]^2$; see \eqref{BCweight} with $w =w_{\a,\b}$ and the remark at the end of Subsection 3.1. 
Both weight functions $\CW_{\a,\b,\g}$ and $B_{\a,\b,\g}$ are defined on $[-1,1]^2$, and they
satisfy
\begin{equation}\label{B-CW}
                   \CW_{\a, -\frac12,\g}(x,  y) = 4^\g B_{\g, \g, \a}(x,y). 
\end{equation}
Consequently, there is some kind of automorphism among these orthogonal polynomials. By Theorem
\ref{thm:QopJacobi}, symmetric orthogonal polynomials with respect to $\CW_{\g, - \frac12,\a}$ 
are given by, for $0 \le k \le n$, 
$$
     P_{k,n}^{\g, -\frac12,\a}(2xy, x^2+ y^2 -1) \quad \hbox{and}\quad 
           (x+ y) P_{k,n}^{\g, \frac12,\a}(2xy, x^2+ y^2 -1) 
$$
of degree $2n$ and $2n+1$, respectively. These are, by \eqref{B-CW}, symmetric orthogonal polynomials
for $B_{\a,\a,\g}$, from which we can derive orthogonal polynomials for $W_{\a,\a,\g}$ by a change of
variables $u = x + y$ and $v = x y$ as shown in the end of Subsection 3.1. Since
$x^2+ y^2 = u^2 - 2 v$, we see that 
$$
     P_{k,n}^{\g, -\frac12,\a}(2 v, u^2 - 2 v -1) \quad \hbox{and}\quad 
                u_1 P_{k,n}^{\g, \frac12,\a}(2 v, u^2 - 2 v -1) 
$$
are orthogonal polynomials with respect to $W_{\a,\a,\g}$ of degree $k+n$. Comparing the leading 
coefficients by \eqref{eq:jacobi-biangle}, we conclude that 
\begin{align*}
  P_{k,n}^{\g, -\frac12,\a}(2 v, u^2 - 2 v -1)&  = 2^{n-k} P_{n-k,n+k}^{\a, \a,\g} (u,v), \\
  u P_{k,n}^{\g, \frac12,\a}(2 v, u^2 - 2 v -1) & = 2^{n-k} P_{n-k,n+k}^{\a, \a,\g} (u,v).   
\end{align*}
These are, however, precisely \eqref{quadraticPkn}. These relations translate to orthogonal 
polynomials with respect to $\CW_{\a,\b,\g}$ on $[-1,1]^2$ as follows. Let 
${}_iQ_{k,n}^{\a,\b,\g}$ denote the orthogonal polynomials given in \eqref{Qeven2} and \eqref{Qodd2} 
but with $P_{k,n}^{\a,\b,\g}$ as the monic orthogonal polynomial as in \eqref{eq:jacobi-biangle}. 

\begin{prop}
We have the following quadratic transforms, for $0 \le k \le \lfloor \frac{n}{2} \rfloor$,  
\begin{align*}
  & {}_1Q_{k,n}^{\g, -1/2,\a} (\cos \t, \cos \phi) 
         = 2^{\lfloor \frac{n}{2} \rfloor -k} {}_1Q_{\lfloor \frac{n}{2} \rfloor-k,2 n- 2 \lfloor \frac{n}{2} \rfloor+2k}^{\a,\a,\g} 
              \left(\cos \tfrac{\t - \phi}{2}, \cos \tfrac{\t + \phi}{2}\right),  \\
  &  \sin \t \sin \phi \, {}_2Q_{k,n}^{\g, -1/2,\a} (\cos \t, \cos \phi)  \\
      & \quad = 2^{\lfloor \frac{n+1}{2} \rfloor -k} \sin \tfrac{\t - \phi}{2} \sin \tfrac{\t + \phi}{2} \times
           {}_2Q_{\lfloor \frac{n-1}{2} \rfloor-k, 2 n- 2 \lfloor \frac{n+1}{2} \rfloor+2k}^{\a,\a,\g} 
              \left(\cos \tfrac{\t - \phi}{2}, \cos \tfrac{\t + \phi}{2}\right). 
\end{align*}
\end{prop}

\begin{proof}
From the quadratic transform formulas satisfied by $P_{k,n}^{\a,\b,\g}(u_1,u_2)$, we have 
\begin{align}  \label{quadraticP}
\begin{split}
   P_{n-k,n+k}^{\a,\a,\g}(x+y,xy) & = 2^{-n+k} P_{k,n}^{\g, - \frac12, \a}(2xy, x^2 +y^2 -1), \\
  (x+y)^{-1} P_{n-k,n+k+1}^{\a,\a,\g}(x+y,xy) & = 2^{-n+k} P_{k,n}^{\g, \frac12, \a}(2 xy, x^2 +y^2 -1). 
\end{split}
\end{align}
If $x = \cos \theta$ and $y = \cos \phi$, then it is easy to see that 
$$
   x+ y = 2 \cos \tfrac{\t - \phi}{2}\cos \tfrac{\t + \phi}{2}, \quad 
        x y = \cos^2 \tfrac{\t - \phi}{2} + \cos^2 \tfrac{\t + \phi}{2} -1
$$
Consequently,  the equations for ${}_1Q_{k,2m}^{\g, -1/2,\a}$ and ${}_1Q_{k,2m+1}^{\g, -1/2,\a}$ 
follow immediately from \eqref{quadraticP}, \eqref{Qeven} and \eqref{Qodd}. For 
${}_2Q_{k,2m}^{\g, -1/2,\a}$ and ${}_2Q_{k,2m+1}^{\g, -1/2,\a}$ we use in addition the 
trigonometric identities $\cos \t  - \cos \phi = 2 \sin \tfrac{\t - \phi}{2} \sin \tfrac{\t + \phi}{2}$ 
and $\cos^2 \tfrac{\t - \phi}{2} - \cos^2 \tfrac{\t + \phi}{2} = \sin \t \sin \phi$. 
\end{proof}

In the case of $\a = \b = -1/2$, the weight function becomes 
$$
   \CW_{-\frac12, - \frac12, \g}(x,y) = (1-x^2)^\g (1-y^2)^\g, \quad (x,y) \in [-1,1]^2,
$$
which is the product Gegenbauer weight function. An orthonormal basis of $\CV_n^d$ for 
this weight function is usually given by the product Jacobi polynomials
$$
      P_{k,n}(x,y) = p_k^{(\g,\g)}(x) p_{n-k}^{(\g,\g)}(y), \quad 0 \le k \le n.    
$$
In this case, another basis for $\CV_n^d$ can be stated as follows: 

\begin{prop}
For $\a = \b = -1/2$, the orthonormal basis for $\CV_{2n}^d(\CW_{-\frac12, - \frac12, \g})$ 
is given by
\begin{align} \label{productEven}
\begin{split}
  & \frac{1}{\sqrt{2}} \left( p_{n+k}^{(\g,\g)}(x) p_{n-k}^{(\g,\g)}(y)+
       p_{n+k}^{(\g,\g)}(y)  p_{n-k}^{(\g,\g)}(x)\right), \quad 0 \le k \le n, \\ 
  & \frac{1}{\sqrt{2}} \left( p_{n+k+1}^{(\g,\g)}(x) p_{n-k-1}^{(\g,\g)}(y) -
       p_{n+k+1}^{(\g,\g)}(y)  p_{n-k-1}^{(\g,\g)}(x)\right), \quad 0 \le k \le n-1, 
\end{split}
\end{align}
and the orthogonal basis for $\CV_{2n+1}^d(\CW_{-\frac12, - \frac12, \g})$ is given by 
\begin{align} \label{productOdd}
\begin{split}
  &   \frac{1}{\sqrt{2}} \left( p_{n+k+1}^{(\g,\g)}(x) p_{n-k}^{(\g,\g)}(y)+
       p_{n+k+1}^{(\g,\g)}(y)  p_{n-k}^{(\g,\g)}(x)\right), \quad 0 \le k \le n, \\ 
   &   \frac{1}{\sqrt{2}} \left( p_{n+k+1}^{(\g,\g)}(x) p_{n-k}^{(\g,\g)}(y) -
       p_{n+k+1}^{(\g,\g)}(y)  p_{n-k}^{(\g,\g)}(x)\right), \quad 0 \le k \le n. 
\end{split}
\end{align}
\end{prop}

\begin{proof}
The orthogonality of these polynomials follows from \eqref{P1/21/2g} and 
Theorem \ref{thm:QopJacobi}. They can also be verified directly. 
\end{proof}

Finally, let us mention that under the change of variables $u = 2 xy$ and 
$v = x^2+y^2 -1$, the operator $E_-^{\a,\b}$ in \eqref{Eab} becomes
\begin{align}
   \CE_-^{\a,\b}: = \frac12 \frac{\partial^2}{\partial x \partial y} + &
        \frac{1}{2(x^2-y^2)} \left[( (\a+\b+1)x + (\a-\b)y ) \frac{\partial}{\partial x} \right. \\
          &  -  \left. ( ( \a - \b)x + (\a+\b+a)y ) \frac{\partial}{\partial y} \right], \notag
\end{align}
which has a simple form for the second order derivatives, so that, by \eqref{Qeven},  
\begin{align*}
   \CE_-^{\a,\b} {}_1Q_{k,2n}^{\a,\b,\g} (x,y) & = -(n-k)(n+k+\a+\b+1) {}_1Q_{k,2 n-2}^{\a,\b,\g+1}(x,y).
\end{align*}
The operator $\CE_-^{\a,\b}$ does not, however, act on ${}_2Q_{k,2n}^{\a,\b,\g}$ 
in the same manner. As the operator $E_+^{\a,\b}$ has the same second order derivatives
as that of $E_-^{\a,\b}$, we can also have an $\CE_+^{\a,\b}$ that has simple second order
derivatives and act on ${}_1Q_{k,2n}^{\a,\b,\g}$ according to \eqref{Eab}. 
  
\subsection{Reproducing kernel}  
We express the reproducing kernel for $K_n(\CW_{\g}; \cdot, \cdot)$, which 
we denote by $\CK_n^{(\g)}(\cdot, \cdot)$ below, in terms of the reproducing kernel
$K_n^{(\g)}(\cdot,\cdot)$ defined in Subsection 3.2. For $W_\g^{(i,j)}$ defined in \eqref{Wij}
with $i,j = 0,1$, we denote by $K_n^{(\g),i,j}(\cdot,\cdot)$ the reproducing kernel 
$K_n(W_\g^{(i,j)}; \cdot,\cdot)$.

\begin{thm}
For $x = (x_1,x_2)$, $y = (y_1,y_2)$, define 
\begin{equation} \label{eq:s-t}
     s = (s_1,s_2) = (2 x_1x_2, x_1^2+x_2^2 -1), \quad t  = (t_1,t_2)= (2 y_1 y_2, y_1^2+ y_2^2-1). 
\end{equation}
Then the reproducing kernel $\CK_n^{(\g)}(\cdot, \cdot)$ for $\CW_{\g}$ 
is given by 
\begin{align} \label{eq:repCW}
    \CK_{n}^{(\g)}(x,y)    =  \,K_{\lfloor \frac{n}{2} \rfloor}^{(\g)}(s,t)  & +  
            d_{w^{(1,1)}}^\g (x_1^2-x_2^2)(y_1^2-y_2^2) K_{\lfloor \frac{n}{2} \rfloor -1}^{(\g), 1,1}(s,t) \\
            & +  d_{w^{(1,0)}}^\g (x_1+x_2)(y_1+y_2) K_{\lfloor \frac{n-1}{2} \rfloor}^{(\g),0,1}(s,t) \notag \\
              &  +   d_{w^{(0,1)}}^\g(x_1-x_2)(y_1-y_2) K_{\lfloor \frac{n-1}{2} \rfloor }^{(\g), 1,0}(s,t), \notag
\end{align}
where $d_{w^{(i,j)}}^\g =  a_{w^{(i,j)}}^\g /a_w^\g$ for $i,j =0,1$. 
\end{thm}

\begin{proof}
We consider $\CK_{2n}^{(\g)}(x,y)$. By the definition of $K_n^{(\g)}(\cdot, \cdot)$ as in 
\eqref{reprodK} and Theorem \ref{thm:Qop}, it follows readily that $\CK_{2n}^{(\g)}(x,y)$
belongs to $\Pi_{2n}^d$ as a function of either $x$ or $y$. To see that it reproduces polynomials 
in $\Pi_{2n}^d$, we verify 
$$
   \la \CK_{2n}^{(\g)}(x, \cdot),  {}_iQ_{k, m}^{(\g)} \ra_{\CW_{g}} =   {}_iQ_{k, m}^{(\g)}, 
           \quad 0 \le k \le m \le 2n, \quad i =1,2,
$$ 
using \eqref{Int-P-Q} and \eqref{IpCW}. For ${}_1Q_{k,2m}^{(\g)}$, this follows immediately from
\eqref{Int-P-Q} and the reproducing property of $K_n^{(\g)}$, since among the four terms in 
the right hand side of \eqref{eq:reprodW-1/2}, only the first term has a non-zero inner product with
${}_1Q_{k,2m}^{(\g)}$ by orthogonality. For ${}_2Q_{k,2m}^{(\g)}$, we use \eqref{Qodd2} and, 
in addition, $d_{w^{(1,1)}}^\g (x_1^2-x_2^2)^2 \CW_{\g}(x_1,x_2)  = W_{\g}^{(1,1)}(u_1,u_2)$. The 
other two cases, ${}_iQ_{k,2m+1}^{\a,\b,\g}$ with $i =1,2$, work out similarly.
\end{proof}

In the case of $\CW_{\a,\b,\g}$, the formula for $\CK_n^{\a,\b,\g}(\cdot,\cdot)$ takes the form 
\begin{align} \label{eq:reprodCW}
    \CK_{n}^{\a,\b,\g}(x,y)  =  \,K_{\lfloor \frac{n}{2} \rfloor}^{\a,\b,\g}(s,t) & +  
        d_{\a,\b,\g}^{(1,1)} (x_1^2-x_2^2)(y_1^2-y_2^2) K_{\lfloor \frac{n}{2} \rfloor -1}^{\a+1,\b+1,\g}(s,t) \\
            &\, + d_{\a,\b,\g}^{(0,1)}   (x_1+x_2)(y_1+y_2) K_{\lfloor \frac{n-1}{2} \rfloor}^{\a,\b+1,\g}(s,t) \notag\\
            &\, + d_{\a,\b,\g}^{(1,0)}  (x_1-x_2)(y_1-y_2) K_{\lfloor \frac{n-1}{2} \rfloor }^{\a+1,\b,\g}(s,t),  \notag
\end{align}
where $d_{\a,\b,\g}^{(i,j)} = a_{\a+i,\b+j,\g}/a_{\a,\b,\g}$ for $i,j = 0,1$. In the case of $\g = \pm \frac12$, 
we can then use Theorem \ref{thm:reprodW} to deduce closed 
formulas for the reproducing kernel $ \CK_{n}^{\a,\b, \pm \frac12}(\cdot,\cdot)$, which take 
simpler forms in the variables 
$$
  ( x_1,x_2) = (\cos \t_1, \cos \t_2) \quad \hbox{and}\quad (y_1,y_2) = (\cos \phi_1, \cos \phi_2).
$$
Indeed, using the relation \eqref{theta-phi} and $(s,t)$ in \eqref{eq:s-t}, it follows from 
\eqref{eq:reprodW-1/2} that
\begin{align} \label{Kn-st-1/2}
    & K_n^{\a,\b,- \frac12}(s,t)  \\
        & \quad = \frac12 \left[  k_n^{\a,\b}(\cos (\t_1-\t_2), \cos (\phi_1 - \phi_2))  
            k_n^{\a,\b}(\cos (\t_1+\t_2), \cos (\phi_1 + \phi_2))  \right.   \notag \\
       & \quad \quad \left.  +  k_n^{\a,\b}(\cos (\t_1-\t_2), \cos (\phi_1 + \phi_2)) 
               k_n^{\a,\b}(\cos (\t_1+\t_2), \cos (\phi_1 - \phi_2))  \right], \notag
\end{align}
and, since $\cos (\t_1- \t_2) - \cos (\t_1 + \t_2) = 2 \sin \t_1\sin \t_2$, 
\begin{align} \label{Kn-st+1/2}
    & K_n^{\a,\b, \frac12}(s,t) = \frac18 (\sin \t_1 \sin \t_2 \sin \phi_1\sin \phi_2)^{-1} \\
        & \quad \times  \left [ k_{n+1}^{\a,\b}(\cos (\t_1-\t_2), \cos (\phi_1 - \phi_2))  
            k_{n+1}^{\a,\b}(\cos (\t_1+\t_2), \cos (\phi_1 + \phi_2)) \right.    \notag \\
       & \quad \quad \left. -  k_{n+1}^{\a,\b}(\cos (\t_1-\t_2), \cos (\phi_1 + \phi_2)) 
               k_{n+1}^{\a,\b}(\cos (\t_1+\t_2), \cos (\phi_1 - \phi_2))\right]. \notag
\end{align}
Substituting \eqref{Kn-st-1/2} and \eqref{Kn-st+1/2} into \eqref{eq:reprodCW} gives a compact
formula of $\CK_n^{\a,\b,\pm \frac12}$ in terms of the reproducing kernels of Jacobi polynomials. 

In the case of $\a = \b = -1/2$, the weight function $\CW_{\a,\b,-\frac12}$ is the product
Chebyshev weight 
$$
    W_{-\frac12,-\frac12,-\frac12}(x_1,x_2) = \frac{1}{\pi^2 \sqrt{1-x_1^2} \sqrt{1-x_2^2}}, \quad (x_1,x_2) \in [-1,1]^2.
$$
Even in this case, the formula \eqref{eq:reprodCW} is new. Previously, another closed formula
for the kernel $K_n(W_0; \cdot,\cdot)$ was given in \cite{X95}. Our new formula, however, is 
more easily adopted for studying convergence of Fourier orthogonal expansions as seen in our
next section.

\section{Fourier Orthogonal Expansions} 
\setcounter{equation}{0}

In this section, we study orthogonal expansions for both $W_{-\frac12}$ on $\Omega$
and $\CW_{ -\frac12}$ on $[-1,1]$. The results include both $L^p$ convergence and the 
uniform convergence. The $L^p$ convergence will be established for $W_{-\frac12}$ and $\CW_{-\frac12}$ 
associated with the generalized Jacobi weight defined in \eqref{GJ}. The uniform convergence 
will be established for  $W_{-\frac12}$ and $\CW_{-\frac12}$ associated with the Jacobi weight. 

We are mainly interested in the case of $\CW_{-\frac12}$, which lives on the square $[-1,1]^2$. 
The study of the $L^p$ convergence of the Fourier orthogonal expansion on the square has been 
lagging behind, perhaps unexpected, the study on the triangle and on the disk. In fact, the $L^p$ 
convergence for the product Jacobi weight on the square has not been established. One reason is 
the lack of a useable formula for the reproducing kernel, which, as we explained in Remark 2.1, does 
not have a product structure. Given this background, our result (see Subsection 5.3) is somewhat 
surprising, as it shows that the case of the weight function 
$$
    W_{\a,\b,-\frac12}(x,y) =  |x-y|^{2 \a+1}|x+y|^{2 \b+1} (1-x^2)^{-\frac12}(1-y^2)^{-\frac12},
$$
or more general $W_{-\frac12}$ in  \eqref{GJ-exp}, can be worked out so much easier than that of the product Jacobi weight function 
on the square.

To get to $\CW_{-\frac12}$ on $[-1,1]$, we need to deal with $W_{-\frac12}$ on $\Omega$ first, which in
return relies on results on $w \in GJ$. In our first subsection we prove some results 
for the generalized Jacobi series of one variable, which are then used to study orthogonal expansions 
for $W_{-\frac12}$ in the second subsection. The results for $\CW_{-\frac12}$ are presented in the third 
subsection. Throughout this section, the constant $c$ will denote a generic constant, its 
value may change from line to line, and we denote the ordinary Jacobi weight function 
by $J_{\a,\b}$, that is, 
$$
J_{\a,\b}(x) := (1-x)^\a (1+x)^\b, \qquad -1 < x< 1. 
$$

\subsection{Orthogonal expansions in generalized Jacobi polynomials}

Let $w$ be a generalized Jacobi weight, $w \in GJ$, defined in \eqref{GJ}. For $1 \le p \le \infty$, 
the $L^p$ norm of $f \in L^p(w,[-1,1])$ is defined by 
$$
    \|f\|_{w,p} = \left( \int_{-1}^1 |f(y)|^p  w (y)dy \right)^{1/p}, \quad 1 \le p < \infty,
$$
and for $ p = \infty$, we replace the $L^p$ space by $C[-1,1]$ with the uniform norm 
$\|f\|_{w, \infty} = \|f\|_{\infty}$. For $1 < p < \infty$ let $q$ be defined by $1/p + 1/q =1$.

Recall the partial sum operator $s_n(w; f)$, see \eqref{eq:partial}, of the orthogonal expansion. 
For proving the mean convergence of $\CS_n(\CW_{-\frac12}; f)$, we need to study the convergence 
of a family of operators closely related to $s_n(w)$, $w \in GJ$. For $i, j \ge 0$, we define   
\begin{align}\label{snij}
  s_n^{i,j} (w; f, x) := J_{\frac{i}{2},\frac{j}{2}}(x)  
        \int_{-1}^1 f(y) k_n(J_{i,j} w; x,y) w(y) J_{\frac{i}2, \frac{j}{2}}(y)dy. 
\end{align}
Evidently, $ s_n(w) =  s_n^{0,0}(w)$. We shall show that these operators have the same
convergence behavior as that of $s_n (w;f)$. 

Standard Hilbert space theory shows that $s_n (w; f)$ converges to $f$ in $L^2(w)$ norm.
The following theorem gives the convergence of $s_n^{i,j}(w; f)$ in $L^p$ space. 

\begin{thm} \label{Jacobi-mean-ij}
Let $w, u,  v \in GJ$. Then for $1 < p < \infty$,
\begin{equation} \label{mean-conv-ij}
   \|s_n^{i,j}(w;f) u \|_{w,p} \le c \|f v\|_{w,p}, \quad  i, j \ge 0.
\end{equation}
for every $f$ such that $\|f v\|_{w,p} < \infty$ if and only if 
\begin{align}\label{mean-cond}
\begin{split}
     u^p w \in L^1, & \qquad u^p  (w J_{\frac12, \frac12})^{-\frac{p}{2}} w \in L^1, \\
     v^{-q} w \in L^1, & \qquad  v^{-q}  (w J_{\frac12,\frac12})^{-\frac{q}{2}} w \in L^1,\\
  \hbox{and}\quad  u(x) \le v(x), & \qquad \hbox{$x \in (-1,1)$}.
\end{split}
\end{align}
In particular, \eqref{mean-conv-ij} implies that $\| (s_n (w; f - f)u\|_{w ,p} \to 0$ when $n\to \infty$
for every $f$ such that $\|f v\|_{w,p} < \infty$.
\end{thm}

\begin{proof}
For $i = j =0$, this result was proved in \cite{X93} (for various earlier results, see \cite{Nevai2} 
and the references therein). We show that the general case of $s_n^{i,j}(w; f)$ 
can  be deduced from the case $i = j = 0$. The operators $s_n^{i,j}(w)$ can be expressed in terms of 
the partial sums of Jacobi series. Let us define
$$
       f_{i,j}(x):=  f(x) / J_{\frac{i}{2}, \frac{j}{2}} (x). 
$$
Then directly from its definition \eqref{snij}, we see that 
\begin{align}  \label{sn-ij}
  s_n^{i,j} (w; f, x)  =  J_{\frac{i}{2}, \frac{j}{2}} (x) \, s_n (J_{i,j}w; f_{i,j}, x).  
\end{align}
The inequality \eqref{mean-conv-ij} is easily seen to be equivalent to, using \eqref{sn-ij}, 
\begin{equation} \label{mean-inequa-ij}
        \|s_n (J_{i,j} w; f) u_{i,j} \|_{J_{i,j} w,p} \le c \|f v_{i,j}\|_{J_{i,j}w,p}
\end{equation}
if we define $u_{i,j}$ and $v_{i,j}$ by  
$$
   u_{i,j}(y) := J_{\frac{i}{2}, \frac{j}{2}}(y) (J_{i,j}(y))^{-\frac{1}{p}} u(y) \quad\hbox{and}\quad
   v_{i,j}(y):=  J_{\frac{i}{2}, \frac{j}{2}}(y) (J_{i,j}(y))^{-\frac{1}{p}} v(y).
$$
The inequality \eqref{mean-inequa-ij} holds, by the result for $i = j = 0$, under the condition 
\eqref{mean-cond} with $w$ replaced by $J_{i,j} w$, $u$ and $v$ replaced by 
$u_{i,j}$ and $v_{i,j}$. We now verify that these conditions hold under \eqref{mean-cond}.
The condition $u_{i,j}(y) \le  v_{i,j}(y)$ holds evidently under $u(y) \le v(y)$ 
of \eqref{mean-cond}. A quick computation shows that 
$$
  u_{i,j}^p J_{i,j} w  = J_{i,j}^{p/2} u^p w, \quad
  v_{i,j}^{-q} J_{1,0}w  = J_{i,j}^{q/2} v^{-q} w,
$$
so that, using $w_{i,j}(y) \le c$, both are $L^1$ functions under \eqref{mean-cond}. A 
similar computation shows that 
\begin{align*}
  u_{i,j}^p (J_{i+\frac12,j+\frac12} w)^{-\frac{p}{2}} J_{i,j}w &\ 
                      = u^p (J_{\frac12,\frac12} w)^{-\frac{p}{2}} w,\\
  v_{i,j}^{-q} (J_{i+\frac12,j+\frac12} w )^{-\frac{q}{2}} J_{i,j}w &\ 
                      = v^{-q} (J_{\frac12,\frac12} w )^{-\frac{q}{2}} w.
\end{align*}
Since the right hand sides of the these two expressions are exactly those appeared in \eqref{mean-cond}, 
all conditions under which \eqref{mean-inequa-ij} hold are verified under \eqref{mean-cond}. This 
establishes \eqref{mean-conv-ij}.
\end{proof}

The special case $u(x) = v(x) \equiv 1$ and $w$ is a multiple of the Jacobi weight is stated below as 
a corollary,  in which the conditions in \eqref{mean-cond} are simplified to \eqref{Jacobi-cond} below. 

\begin{cor}
Let $w(x) = \psi(x) (1-x)^\a (1+x)^\b$, where $\a, \b > -1$ and $\psi$ is as in \eqref{GJ}, and let $i,j \ge 0$. 
Then 
$$
         \|s_n^{i,j}(w; f)\|_{w, p} \le c \|f\|_{w,p} 
$$
for every $f \in L^p(w, [-1,1])$ if and only if 
\begin{equation} \label{Jacobi-cond}
        2 - \frac{2}{2\max\{\a,\b\} +3} < p < 2 +\frac{2}{2\max\{\a,\b\} +1}.
\end{equation}
\end{cor}

The Theorem \ref{Jacobi-mean-ij} settles the mean convergence of $s_n^{i,j}(w;f)$ for 
$1 < p < \infty$. For the cases of $p =1$ or $p =\infty$, it is easily seen that 
\begin{align} \label{Lebsgue}
  \|s_n^{i,j}(w) \|_{\infty} & =  \|s_n(w)\|_{w, 1} \\
    & = \max_{x\in [-1,1]}
         J_{\frac{i}{2},\frac{j}{2}}(x)  \int_{-1}^1 |k_n ( J_{i,j}w; x,y)| J_{\frac{i}{2}, \frac{j}{2} }(y) w(y) dy. \notag
\end{align}
In fact, the proof in the case that $i=j = 0$ is classical and it carries over just as well in the 
general case. We shall determine the order of $s_n^{i,j}$ for the classical Jacobi weight $w = J_{\a,\b}$ and 
denote, for simplicity,   
$$
    s_n^{(\a,\b),i,j} f  := s_n^{i,j} (J_{\a,\b}; f). 
$$ 
By definition, $s_n^{(\a,\b)}  =  s_n^{(\a,\b),0,0}$ is the partial sum of the classical Jacobi series. The 
quantity $\|s_n^{(\a,\b)}\|_{\infty}$, sometimes called the Lebesgue constant, determines the convergence
behavior of $s_n^{(\a,\b)}f$ when $p =1$ and $p = \infty$. 

The asymptotic order of $\|s_n^{(\a,\b)}\|_{\infty}$ is usually determined by using the convolution 
structure of the Jacobi series, which shows that the maximum 
$$
\|s_n^{(\a,\b)} \|_{\infty}  = \max_{x\in [-1,1]}
          \int_{-1}^1 |k_n^{(\a,\b)}(x,y)|  J_{\a,\b} (y) dy
$$
is attained at the point $x = 1$. The same scheme, however, does not apply to $\|s_n^{(\a,\b), i,j}\|_\infty$ 
if either $i > 0$ or $j > 0$ because of the factor $(1-x)^{i/2} (1+x)^{j/2}$ in front. Nevertheless, the result 
still holds and it can be proved by using a sharp estimate of the kernel function of $s_n^{(\a,\b),i,j} (f)$ given by  
\begin{align*}
   k_n^{(\a,\b),i,j}(\cos \t,\cos \phi) &:= (1-x)^{\frac{i}{2}} (1+ x)^{\frac{j}{2}}(1-y)^{\frac{i}{2}} (1+ y)^{\frac{j}{2}}
         k_n^{(\a+i,\b+j)}(x,y) \\ 
    &  = (\sin \tfrac{\t}{2}\sin \tfrac{\phi}{2})^i (\cos \tfrac{\t}{2} \cos \tfrac{\phi}{2})^j 
         k_n^{(\a+i, \b+j)}(\cos \t, \cos\phi). 
\end{align*}
Evidently,  $k_n^{(\a,\b),0,0}(\cdot,\cdot) =  k_n^{(\a,\b)} (\cdot, \cdot)$. It turns out that the kernels in
this family have the same upper estimate. 
 
\begin{lem}
Let $\a,\b \ge  -1/2$ and $i,j \ge 0$. Then
\begin{align} \label{kernel-esti}
  | k_n^{(\a,\b),i,j}(\cos \t, \cos \phi) |  \qquad\qquad\qquad\qquad \qquad\qquad\qquad\qquad \qquad\qquad\qquad  \\
  \le c   \frac{(\sin \tfrac{\t}{2}\sin \tfrac{\phi}{2} +  n^{-1} |\t-\phi| + n^{-2})^{-\a- \frac12}   
 (\cos \tfrac{\t}{2}\cos \tfrac{\phi}{2} + n^{-1} |\t-\phi| + n^{-2})^{-\b  - \frac12}}{|\t - \phi| + n^{-1}}. \notag  
\end{align}
\end{lem}

\begin{proof}
In the case of $(i,j) = (0,0)$, the estimate is derived from \cite[Theorem 2.7]{DaiX} by setting $\d = 0$ 
and $d =1$, changing variable from $[0,1] \times [0,1]$ to $[-1,1] \times [-1,1]$ and then to $(\cos \t, \cos \phi)$ 
in $[0,\pi] \times [0,\pi]$, and applying elementary trigonometric identities. There are two terms in the 
estimate in \cite[Theorem 2.7]{DaiX} but it is not hard to see that the second term is dominated by the 
first one when $\d = 0$ and $d =1$. 

For the general case of $i, j \ge 0$, applying \eqref{kernel-esti} to $ k_n^{(\a+i, \b+j)}(\cos \t, \cos\phi)$
and using the factors $(\sin \tfrac{\t}{2}\sin \tfrac{\phi}{2})^i$ and $(\cos \tfrac{\t}{2} \cos \tfrac{\phi}{2})^j$ 
to reduce the exponent $- (\a + i + 1/2)$ to $- (\a + 1/2)$  and the exponent $- (\b + j + 1/2)$ to $- (\b + 1/2)$,
we see that $k_n^{(\a,\b),i,j}(\cos \t,\cos \phi)$ has exactly the same upper bound as the right hand 
side of \eqref{kernel-esti}. Consequently, $ \|s_n^{(\a,\b),i,j}\|_\infty$ has the same upper bound. 
\end{proof}

\begin{thm} \label{Jacobi-uniform}
Let $\a,\b \ge  -1/2$ and $i,j \ge 0$. Then
\begin{equation} \label{uniform-Jacobi}
      \|s_n^{(\a,\b),i,j}\|_{\infty} = \CO(1) \begin{cases} n^{\max\{\a,\b\} +1/2}, & \max\{\a,\b\} > -1/2, \\
              \log n, & \max\{\a,\b\} = -1/2. \end{cases}
\end{equation}
\end{thm}

\begin{proof}
The case $i = j = 0$ can be deduced from the convolution structure of the Jacobi series and 
the Lebesgue function at $x =1$ (\cite[Section 9.41]{Sz}). Our proof below uses the kernel 
estimate in \eqref{kernel-esti} and works for the general case of $i,j \ge 0$. For $i, j \ge 0$, we can rewrite the norm  in 
\eqref{Lebsgue} as 
\begin{align*}
     \|s_n^{(\a,\b),i,j}\|_\infty & = \max_{x \in [0,1]} \int_{-1}^1 \left | k_n^{(\a,\b),i,j} (x,y) \right | w_{\a,\b}(y) dy  \\
       & =  \int_0^\pi  \left | k_n^{(\a,\b),i,j} (\cos \t, \cos \phi) \right |
                  (\sin \tfrac{\phi}{2})^{2\a+1} (\cos \tfrac{\phi}{2})^{2\b+1} d\phi. 
\end{align*}
By symmetry, we consider only $ 0 \le \t \le \pi/2$. If $ \frac{3 \pi}{4} \le \phi \le \pi$, then $|\t - \phi| \ge \pi /4$,
so that, by the estimate of the kernel, $|k_n^{(\a,\b)}(\cos \t, \cos \phi)| \le c$ as $\a+ \frac12 \ge 0$ and 
$\b + \frac 12 \ge 0$, so that  the integral over $[3 \pi /4, \pi]$ is bounded. On the other hand, if 
$0 \le \phi \le 3 \pi/4$, then $\cos \frac{\t}{2} \cos \frac{\phi}{2} \sim 1$, so that the estimate of the 
kernel shows that 
$$
 | k_n^{(\a,\b)}(\cos \t, \cos \phi) | \le c \frac{(\sin \tfrac{\t}{2}\sin \tfrac{\phi}{2} +  n^{-1} |\t-\phi| + n^{-2})^{-\a- \frac12}}
     {|\t - \phi| + n^{-1}}.
$$
The case $\a = -1/2$ is easier, let us assume $\a > -1/2$. We then divided the integral into three terms, 
$$
    \int_0^{3 \pi /4} \cdots d\phi = \int_0^{\t/2}\cdots  d\phi  + \int_{\t /2}^{3 \t/2} \cdots  d\phi+ 
        \int_{3 \t/2}^{3 \pi/4}\cdots  d\phi .
$$ 
Using the estimate of the kernel, these three integrals can be shown to be bounded by $c \, n^{\a+1/2}$ 
by making use of the following facts: For the integral over $[0, \t/2]$, $|\phi - \t| \sim \t$; for the integral 
over $[\t/2, 3 \t/2]$, $\t \sim \phi$; for the integral over $[3 \t/2, 3 \pi /4]$, $|\t - \phi | \ge \phi /3$.  We
leave the details to interested readers. 
\end{proof}

It should be remarked that the asymptotic order of $\|s_n^{(\a,\b)}\|_\infty$ is established for 
$\a, \b > -1$. The reason that we assume $\a, \b \ge -1/2$ lies in the fact that the estimate
of the kernel in \cite{DaiX} was established under this assumption. We expect that the result
extends to 
$$
 \|s_n^{(\a,\b),i,j}\|_\infty = \CO(1) \log n, \quad \max \{\a,\b\} \le - 1/2. 
$$
In fact, our proof already shows this estimate if both $i, j \ge 1$. Only the cases $i = 0$ or 
$j = 0$ remain. What is of interest is to extend the kernel estimate \eqref{kernel-esti}, or in
some modified form, to the range $\max \{\a,\b\} < - 1/2$. 

The reason that we restrict to the classical Jacobi weight function in our estimate of $\|s_n^{i,j}(w)\|_\infty$
is again the lack of a pointwise estimate for the kernel function. 

\subsection{Orthogonal expansions for $W_{-\frac12}$ on $\Omega$}
Let $W$ be a weight function defined on $\Omega \subset \RR^2$. We denote by $L^p(W)$ the
$L^p$ space of functions for which the norm $\|f\|_{W,p}$ is finite, where
$$
     \|f\|_{W,p} := \left( \int_\Omega |f(x)|^p W(x) dx\right)^{1/p}, \qquad 1 \le p < \infty, 
$$
and, for $p  = \infty$, we replace $L^p(W)$ by $C(\Omega)$, the space of continuous functions 
on $\Omega$ with the uniform norm $\|f\|_{W,\infty} = \|f\|_\infty$ on $\Omega$. 

In this subsection we consider the convergence of $S_n(W_{- \frac12}; f)$ with, see \eqref{Wgamma},
$$
    W_{-\frac12} (u,v) = 2 a_w w(x) w(y) |u^2- 4 v|^{-\frac12}, \qquad w \in GJ, 
$$
where $u=x+u$, $v = x y$ and $a_w$ is a normalization constant. 
Because of what we will need in the following subsection, we also define a family of operators as 
follows: For $i,j\ge 0$, 
\begin{align} \label{Sn_ij}
  S_n^{i,j}(W_{-\frac12}; f, x) := \int_\Omega f(y) K_n^{i,j} (W_{-\frac12}; x,y)W_{ -\frac12}(y)dy,
\end{align}  
where 
\begin{align} \label{Kn_ij}
 K_n^{i,j}(W_{-\frac12}; x,y):= J^*_{\frac{i}{2},\frac{j}{2}}(x)J^*_{\frac{i}{2},\frac{j}{2}}(y) 
       K_{n} (J^*_{i,j} W_{-\frac12} ;x,y)
\end{align}
and
$$
     J^*_{\a,\b}(x) : =  (1-x_1+x_2)^{\a}(1+x_1+x_2)^{\b}. 
$$
Evidently, $S_n^{0,0}(W_{-\frac12}; f) = S_{n}(W_{-\frac12}; f)$. We will show that operators 
in this family have the same convergence behavior.
 
Let $u, v$ be the generalized Jacobi weight functions. We define weight functions $U$ and $V$ by
\begin{equation} \label{U-V}
       U(x+y, xy) = u(x)u(y) \quad\hbox{and}\quad  V(x+y,xy)  = v(x)v(y)
\end{equation}
for $-1 < x < y <1$.  The functions $U$ and $V$ are well defined on $\Omega$. 

\begin{thm} \label{thm:meanWab}
Let $w, u, v \in GJ$ and $i,j \ge 0$. Then for $1 < p < \infty$,
\begin{equation} \label{mean-Wab}
      \| S_n^{i,j}(W_{-\frac12}; f) U \|_{W_{  - \frac12},p} \le c \|f V \|_{W_{- \frac12},p}
\end{equation}
for all $f$ such that $\|f V \|_{W_{- \frac12},p}< \infty$ if $u, v, w$ satisfy the conditions in
\eqref{mean-cond}. In particular, \eqref{mean-Wab} implies that $\|(  S_n (W_{-\frac12}; f) - f)V\|_{W_{- \frac12},p}
\to 0$ when $n \to \infty$ for all $f$ such that $\|f V \|_{W_{- \frac12},p} < \infty$. 
\end{thm}
 
\begin{proof}
For $f$ defined on $\Omega$ we define $F(x_1,x_2) := f(x_1+x_2, x_1x_2)$ for $(x_1,x_2) \in [-1,1]^2$. 
We first consider the case $i = j =0$. Let $W^*(x) := w(x_1)w(x_2)$. By Theorem \ref{thm:reprodW}, we have 
\begin{equation*}
   S_n (W_{-\frac12}; f,  x_1+x_2,x_1x_2) = \frac12 \left[S_{n,n}(W^*; F, x_1,x_2)
     + S_{n,n}(W^*; x_2,x_1) \right], 
\end{equation*}
where $S_{n,n} (W^*; F)$ denotes the partial sum of the product generalized Jacobi series that has degree
$n$ in each of the variables $x_1$ and $x_2$. By definition,
\begin{align*}
   S_{n,n}  (W^*; F,  x_1,x_2)     = \int_{-1}^1 \int_{-1}^1 F(y_1,y_2) k_n(w; y_1)k_n(w; y_2)
          W_{\a,\b}^*(y_1,y_2)  dy_1 dy_2.
\end{align*}
Similarly, define $U^*(y_1,y_2): = u(y_1)u(y_2)$ and $V^*(y_1,y_2): = v(y_1)v(y_2)$. Then, applying 
\eqref{Para-Square} with $\g = -\frac12$, we obtain upon using the definition of $W_{-\frac12}$, 
$$
  \|S_n (W_{-\frac12}; f) U \|_{W_{-\frac12},p} \le   \|S_{n,n}(W^*; F) U^*\|_{W^*,p}.
$$
where the norm of the right hand side is taken  over $[-1,1]^2$ against the weight function 
$W^*$. Setting $t_n (w; x_1,y_2) := s_n(w; F(\cdot, y_2), x_1)$, we can write 
$$
    S_{n,n} (W^*; F, x_1,x_2) = s_n(w; t_n (w;  x_1, \cdot), x_2). 
$$
Thus, the product nature of $S_{n,n}(W^*)$ allows us to apply Theorem \ref{Jacobi-mean-ij} 
twice to conclude that 
$$
     \|S_{n,n} (W^*; F) U^*\|_{W^*,p} \le c \|F V^*\|_{W^*,p}  = c  \|f V \|_{W_{ - \frac12},p},
$$
where the equality follows from \eqref{Para-Square}. This completes the proof when $(i,j) = (0,0)$. 

The above proof carries over to the case $i, j > 0$. Indeed, since $(x_1,x_2) \mapsto (x_1+x_2,x_1x_2)$ sends 
$1 \pm x_1+x_2$ to $(1 \pm x_1)(1\pm x_2)$, by Theorem \ref{thm:reprodW}, we have 
\begin{equation*}
  S_n (W_{-\frac12}; f, x_1+x_2,x_1x_2) = \frac12 \left[S_{n,n}^{i,j} (W^*; F, x_1,x_2)+ S_{n,n}^{i,j} (W^*; F, x_2,x_1) \right], 
\end{equation*}
where $S_{n,n}^{i,j} (W^*;F)$ can be expressed in terms of $s_n^{i,j}(w)$ at \eqref{sn-ij} exactly as 
$S_{n,n}(W^*; F)$ is expressed in terms of $s_n(w)$. Consequently, the same proof 
applies and the desired result follows from Theorem \ref{Jacobi-mean-ij}.
\end{proof}

In the case of $p = 1$ or $\infty$, we restrict to the case of $W_{-\frac12} = W_{\a,\b,-1/2}$.

\begin{thm} \label{thm:uniformWab}
Let $\a,\b \ge - \frac12$ and $i,j \ge 0$. Then
\begin{align} \label{uniformWab}
 \|S_n^{i,j}(W_{\a,\b, - \frac12})\|_\infty  & = \|S_n^{i,j}(W_{\a,\b, -\frac12})\|_{W_{\a,\b, - \frac12},1} \\
   &  = \CO(1) \begin{cases} n^{\max\{\a,\b\} +1/2}), & \max\{\a,\b\} > -1/2, \\
                 \log n, & \max\{\a,\b\} = -1/2. \end{cases} \notag
\end{align}
\end{thm}

\begin{proof}
The standard argument shows that 
\begin{align*} 
& \|S_n^{i,j}(W_{\a,\b, - \frac12})\|_\infty    = \|S_n^{i,j}(W_{\a,\b, - \frac12})\|_{W_{\a,\b, - \frac12},1} \\
    &    = \max_{x \in \Omega} (1-x_1+x_2)^{\frac{i}{2}}(1+x_1+x_2)^{\frac{j}{2}}
        \int_{\Omega} |K_n(W_{\a+i,\b+j,-\frac12};x,y)| W_{\a+\frac{i}{2},\b+\frac{j}{2},-\frac12}(y)dy.
\end{align*}
Applying \eqref{Para-Square}, \eqref{eq:reprodW-1/2} and \eqref{sn-ij}, it follows readily that 
$$
  \|S_n^{i,j}(W_{\a,\b, - \frac12})\|_\infty  \le \left(\|s_n^{(\a,\b),i,j}\|_\infty \right)^2
$$
so that the stated result follows from \eqref{uniform-Jacobi}. 
\end{proof} 
 
The case of $S_n(W_{\frac12})$ is harder to work with because of the denominator $x - y$ in
its kernel \eqref{eq:reprodW+1/2}. The proof of Theorem \ref{thm:meanWab} could carry over 
only if we modified the definitions of $U$ and $V$ to include a power of $|x-y|$. Since it adds 
little to our understanding, we shall not pursuit it any further.  

\subsection{Orthogonal expansions for $\CW_{ -\frac12}$ on $[-1,1]^2$}
We now consider the convergence of the Fourier orthogonal expansions with respect to
$\CW_{-\frac12}$ on $[-1,1]^2$. 

Let $u, v \in GJ$ be the generalized Jacobi weights. We define the weight functions $\CU$ and $\CV$ 
on $[-1,1]^2$ by
\begin{align*}
  \CU(\cos \t, \cos \phi) := & u(\cos (\t - \phi)) u(\cos (\t+\phi)), \\
  \CV(\cos \t, \cos \phi) := & v(\cos (\t - \phi)) v(\cos (\t+\phi)).
\end{align*}

\begin{thm} \label{thm:meanCWab}
Let $w, u, v \in GJ$. Then for $1 < p < \infty$, 
\begin{equation} \label{mean-CWab}
      \|\CS_n(\CW_{- \frac12}; f) \CU \|_{\CW_{ - \frac12},p} \le c \|f \CV \|_{\CW_{- \frac12},p}
\end{equation}
for all $f$ such that $\|f \CV \|_{\CW_{- \frac12},p}< \infty$ if $u, v, w$ satisfy the conditions in
\eqref{mean-cond}. In particular, \eqref{mean-CWab} implies that 
$\|( \CS_n(\CW_{ -\frac12}; f) - f)\CV\|_{\CW_{ - \frac12},p} \to 0$ when $n \to \infty$ for all $f$ 
such that $\|f \CV \|_{\CW_{ - \frac12},p} < \infty$. 
\end{thm}

\begin{proof}
We consider the case $\CS_{2n}(\CW_{- \frac12}; f)$. For a given $f$ on $[-1,1]^2$, we define
$f^*$ by 
$$
        f^*(2 x_1x_2, x_1^2+x_2^2-1) = f(x_1,x_2). 
$$
By \eqref{theta-phi}, $f^*$ is well defined on $\Omega$. 
If $t_1  = 2 y_1y_2$ and $t_2 = y_1^2+y_2^2-1$, then it is easy to see that 
$1-t_1+t_2= (y_1-y_2)^2$ and $1+t_1+t_2= (y_1+y_2)^2$. Hence, by \eqref{eq:repCW}
and \eqref{Kn_ij}, we obtain that  
\begin{align*} 
   \CK_n(\CW_{-\frac12}; x,y) = &  K_{\lfloor \frac{n}{2} \rfloor}(W_{-\frac12}; s,t)  
              +  d_{0,1} K_{\lfloor \frac{n-1}{2} \rfloor}^{0,1}(W_{-\frac12}; s,t)  \\
              & +  d_{1,1} K_{\lfloor \frac{n-1}{2} \rfloor}^{1,0}( W_{-\frac12};s,t)   
              +  d_{1,1} K_{\lfloor \frac{n}{2} \rfloor -1}^{1,1}(W_{-\frac12}; s,t).
\end{align*}
Consequently, by \eqref{Int-P-Q}, it follows that 
\begin{align} \label{CS-S}
 \CS_{2n}(\CW_{-\frac12}; f,x ) = S_{n}(W_{-\frac12}; f^*, s)& +S_{n-1}^{1,0}(W_{-\frac12}; f^* , s) \\
     & + S_{n-1}^{0,1}(W_{-\frac12} ; f^*, s) + S_{n-1}^{1,1}(W_{-\frac12} ; f^*, s), \notag
 \end{align} 
where $s = ( 2 x_1 x_2, x_1^2+x_2^2-1)$. Recall the definition of $U$ and $V$ in
\eqref{U-V}. A simple computation via elementary trigonometric identities shows that 
$$
   \CU(x_1,x_2)  = U(2x_1x_2, x_1^2 + x_2^2 -1) \quad \hbox{and}\quad  
           \CV(x_1,x_2)  = V(2x_1x_2, x_1^2 + x_2^2 -1).
$$
Consequently, applying \eqref{Int-P-Q} again, we see that 
\begin{align*}
\|\CS_{2n}(\CW_{-\frac12}; f) \CU\|_{\CW_{-\frac12},p}    \le &
     \| S_{n}(W_{-\frac12}; f^*)U\|_{W_{-\frac12},p}  +
      \|S_{n-1}^{1,0}(W_{ - \frac12}; f^*)U\|_{W_{-\frac12},p} \\
     & + \|S_{n-1}^{0,1}(W_{-\frac12}; f^*)U\|_{W_{-\frac12},p} 
        + \| S_{n-1}^{1,1}(W_{-\frac12}; f^*)U\|_{W_{-\frac12},p}\\
    \le & c \|f^* V\|_{W_{-\frac12},p} = c  \|f \CV\|_{\CW_{-\frac12},p} 
\end{align*}
by Theorem \ref{thm:meanWab} with $i,j = 0,1$. The proof for $\CS_{2n+1}(W_{-\frac12}f)$ follows
analogously. 
\end{proof}

\begin{cor}
Let $\psi$ be a positive, continuously differentiable on $[-1,1]$ and define 
$\Psi(x,y)$ by $\Psi(\cos \t, \cos \phi) = \psi(\cos(\t-\phi), \cos (\t + \phi))$. Then for  
$$
\CW_{\a,\b}(x,y) : = \Psi(x, y) |x-y|^{2\a+1} |x+y|^{2\b+1} (1-x^2)^{-\frac12}(1-y^2)^{-\frac12}, 
$$
where $\a, \b > -1$, 
$$
      \|\CS_n (\CW_{\a,\b}; f) \|_{\CW_{\a,\b},p} \le c \|f \|_{\CW_{\a,\b},p}
$$
for all $f$ such that $\|f\|_{\CW_{\a,\b},p}< \infty$ if \eqref{Jacobi-cond} holds. 
\end{cor}

Even for the case $\a = \b = -1/2$, this corollary is new. In the case of $\a = \b = - \frac12$
and $\psi(x) =1$, i.e., the product Chebyshev weight function, this was established in \cite{X96} 
by identifying the orthogonal expansion with the double Fourier series and $S_n f$ with the 
$\ell_1$ partial sum of the double Fourier series, then applying several results in 
the Fourier analysis. It is worth mentioning that no analogous results are known for the product
Jacobi weight $J_{\a,\b}(x)J_{\a,\b}(y)$ on the square. 

For the norm with $p = 1$ or $\infty$, we go back to the weight function associated with the 
Jacobi weight. 

\begin{thm} \label{thm:uniformCWab}
Let $\a,\b \ge -1/2$. Then
\begin{align} \label{uniformCWab}
 \|S_n(W_{\a,\b, - \frac12})\|_\infty  & = \|S_n(W_{\a,\b, -\frac12}) \|_{\CW_{\a,\b, - \frac12},1} \\
&  = \CO(1) \begin{cases} n^{\max\{\a,\b\} +1/2}), & \max\{\a,\b\} > -1/2, \\
              \log n, & \max\{\a,\b\} = -1/2. \end{cases} \notag
\end{align}
\end{thm}

\begin{proof}
This is an immediate consequence of \eqref{CS-S} and Theorem \ref{thm:uniformCWab}. 
\end{proof}
 
 \medskip
 \noindent
 {\it Acknowledgment}. The author thanks an anonymous referee for his careful review.

\end{document}